\documentclass[11 pt,a4paper]{article}
\usepackage{amsmath,amscd}
\usepackage{amstext}
\usepackage{amssymb,epsfig}
\usepackage{amsthm}
\usepackage{amsfonts}

\usepackage{graphicx}
\usepackage{psfrag}
\usepackage[all]{xy}
\usepackage{multicol}

\renewcommand{\footnote}{\endnote}
\newtheorem{theorem}{Theorem}
\newtheorem{lemma}[theorem]{Lemma}
\newtheorem{proposition}[theorem]{Proposition}
\newtheorem{definition}[theorem]{Definition}
\newtheorem{remark}[theorem]{Remark}
\newtheorem{corollary}[theorem]{Corollary}
\newtheorem{example}[theorem]{Example}

\newtheorem*{thm}{Theorem}
\usepackage[utf8]{inputenc}      

\title{Outer billiard outside regular polygons}

\date{}

\author{Nicolas Bedaride\footnote{ Laboratoire d'Analyse Topologie et Probabilités  UMR 6632 , Université Paul Cézanne, avenue escadrille Normandie Niemen 13397 Marseille cedex 20, France. nicolas.bedaride@univ-cezanne.fr}\and  Julien Cassaigne\footnote{Institut de Mathématiques de Luminy, UMR 6206, Université de la méditérranée, 13288 Marseille. cassaigne@iml.univ-mrs.fr}}

\begin{document}
\maketitle
\begin{abstract}We consider the outer billiard outside a regular convex polygon. We deal with the case of regular polygons with $3,4,5,6$ or $10$ sides. We describe the symbolic dynamics of the map and compute the complexity of the language.
\end{abstract}

\tableofcontents

\bigskip

\section{Introduction}
An outer billiard table is a compact convex domain $P$. Pick a point $M$ outside $P$. There are two half-lines starting from $M$ tangent to $P$; choose one of them, say the one to the right from the view-point of $M$, and reflect $M$ with respect to the tangency point. One obtains a new point, $y$, and the transformation 
$T : TM=y$ is the outer (a.k.a. dual) billiard map, see Figure \ref{fig-def-1}. 
The map $T$ is not defined if
the support line has a segment in common with the outer billiard table. In
the case where $P$ is a convex polygon, the set of points for which $T$ or any of its
iterations is not defined is contained in a countable union of lines and has zero measure.
The dual billiard map was introduced by Neumann in \cite{Neu.59} as a toy model for planet orbits. 
Then the first result has been proved by Moser in \cite{Mos.73}. It concerns the outer billiard outside a closed 
convex smooth curve on the plane, and he proves that the orbits are always bounded using KAM theory. 
In the case of a polygon the same problem has no general solution. A particular class of polygons has been 
introduced by Kolodziej et al. in several articles, see \cite{Viv.Shai.87,Kol,Gu.Si.92}. This class is named the quasi-rational polygons and 
contains all the regular polygons. They prove that every orbit outside a polygon in this class is bounded. Recently Schwartz described a family of polygons for which there exist unbounded orbits, see \cite{Sch.07}. 
In the case of the regular pentagon Tabachnikov completely described the dynamics of the outer billiard map in terms of symbolic dynamics, see \cite{Ta.95}. 
The study of the symbolic dynamics of these map outside a polygon is just at the beginning. By a result of Buzzi, see \cite{Buz.01}, we know that 
the topological entropy is zero, thus the complexity growths as a sub-exponential. Recently Gutkin andTabachnikov proved that the complexity function in the case of a quasi-rational polygon is bounded by a polynomial, see \cite{Gu.Ta.06}.

In this paper we consider the outer billiard map outside regular polygons, and analyze the symbolic dynamics attached to
this map. We are interested in the cases where the polygon has $3,4,5,6$ or $10$ sides and give a complete description of the 
dynamics. Moreover we compute the global complexity of these maps, and generalize the result of Gutkin and Tabachnikov, see \cite{Gu.Ta.06}.


\section{Overview of the paper}
In Section \ref{secdual} we define the outer billiard map outside a polygon. In Section \ref{secsymb} we recall the basic definitions of word combinatorics and explain the partition associated to the dual billiard map. In Section \ref{secsimplif} we simplify our problem, then in Section \ref{secback} we recall the different results on the subject, and in Section \ref{secresult} we can state our results. 

After this point the proof begins with a general method which will be useful during all the rest. The following two sections will be focused on the case of the square, the triangle and the regular hexagon.
In Section \ref{secinduc} we use an induction method for these polygons to describe their languages. In Section \ref{seccalcomp} we finish the computation of the complexity.

Now it remains to study the regular pentagon. First in Section \ref{sectab} we recall some facts on piecewise isometries. Then in Section \ref{secindupent} we use an induction method to describe the dynamics of the dual billiard map. Finally in Sections \ref{seclangpent}, \ref{seccalccomppent}  and \ref{secfin} we describe the language of the dual billiard map outside the regular pentagon and finish the proof. Section \ref{secdeca} finishes the paper with a similar result for the regular decagon. 


\section{Outer billiard}\label{secdual}
We refer to \cite{Ta} or \cite{Gu.Si.92}.
We consider a convex polygon $P$ in $\mathbb{R}^2$ with $k$ vertices. Let $\overline{P}=\mathbb{R}^2\setminus P$ be the complement of $P$.\\
We fix an orientation of $\mathbb{R}^2$. 

\begin{definition}\label{defdual}
Denote by $\sigma_1$ the union of straight lines containing the sides of $P$. We consider the central symmetries $s_i, i=0\dots k-1$ about the vertices of $P$. Define $\sigma_n$, where $n\geq 2$ is an integer, by $\sigma_n=\displaystyle\bigcup_{i=0}^{k-1} s_i\sigma_{n-1}$.
Now the singular set is defined by:
$$Y=\bigcup_{n=1}^\infty\sigma_n.$$
\end{definition}

For any point $M\in\overline{P}\setminus Y$, there are two half-lines $R,R'$ emanating from $M$ and tangent to $P$, see Figure 1. Assume that the oriented angle $R,R'$ has positive measure. Denote by $A^+,A^-$ the tangent points on $R$ respectively $R'$.

\begin{definition}\label{point}
The outer billiard map is the map $T$ defined as follows:
$$\begin{array}{ccccc}
T&:&\overline{P}\setminus Y&\to& \overline{P}\setminus Y\\
&&M&\mapsto&s_{A^+}(M)\end{array}$$ 
where $s_{A^+}$ is the reflection about $A^+$.
\end{definition}

\begin{remark}
This map is not defined on the entire space. The map $T^n$ can be defined on $\overline{P}\setminus \sigma_n$, but on this set $T^{n+1}$ is not everywhere defined. The definition set $\overline{P}\setminus Y$ is of full measure in $\overline{P}$.
\end{remark}

\begin{figure}[h]
\begin{center}
\includegraphics[width= 5cm]{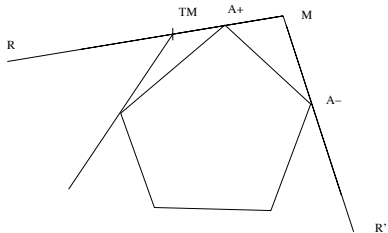}
\psfrag*{A}{$A^{+}$}
\psfrag*{B}{$A^{-}$}
\psfrag*{C}{$\mathbb{R}$}
\psfrag*{D}{$R'$}
\caption{Outer billiard map}\label{fig-def-1}
\end{center}
\end{figure}

Two important families of polygons have been defined in the study of this map: the rational polygons and the quasi-rational polygons.
\begin{definition}
A polygon $P$ is said to be rational if the vertices of $P$ are on a lattice of $\mathbb{R}^2$.
\end{definition}
The definition of the quasi-rational polygons is more technical and we will not need it here. We just mention the fact that every regular polygon is a quasi-rational polygon.


\section{Symbolic dynamics}\label{secsymb}
\subsection{Definitions}
\subsubsection{Words}

For this paragraph we refer to \cite{Pyt}.

\begin{definition}
Let $\mathcal{A}$ be a finite set called the alphabet, a word is a finite string of elements in $\mathcal{A}$. If $v$ is the word $v=v_0\dots v_{n-1}$ then $n$ is called the length of the word $v$. The set of all finite words over $\mathcal{A}$ is denoted $\mathcal{A}^*$. 
\end{definition}

\begin{definition}
A (one sided) sequence of elements of $\mathcal{A}$, $u=(u_n)_{n\in\mathbb{N}}$ is called an infinite word.

A word $v=v_0\dots v_{r-1}$ is said to occur at position $m$ in an infinite word $u$ if there exists an integer $m$ such that for all $i\in[0;r-1]$ we have $u_{m+i}=v_i$. We say that the word $v$ is a factor of $u$. 

For an infinite word $u$, the language of $u$ (respectively the language of length $n$)  is the set of all words (respectively all words of length $n$) in $\mathcal{A}^*$ which appear in $u$. We denote it $L(u)$ (respectively $L_{n}(u)$). 
\end{definition}

\begin{definition}
The shift map is the map defined on $\mathcal{A}^\mathbb{N}$ by $S(u)=v$ with $v_n=u_{n+1}$ for all integer $n$.

The dynamical system associated to an infinite word $u$ is the system $(X_u,S)$ where $S$ is the shift map and $X_u=\overline{\{S^n(u),n\in\mathbb{N}\}}$.

An infinite word $u$ is said to be recurrent if every factor occurs infinitely often. 
\end{definition}

\begin{remark}\label{remcombinmot}
Remark that $u$ is recurrent is equivalent to the fact that $S$ is onto on $X_u$. Moreover we have equivalence between $\omega\in X_u$ and $L_\omega\subset L_u$. Thus the language of $X_u$ is equal to the language of $u$.
\end{remark}

In the following, we will deal with several infinite words, thus we need a general definition of language.

\begin{definition}
A language $L$ is a sequence $(L_n)_{n\in\mathbb{N}}$ where $L_n$ is a finite set of words of length $n$ such that for any word $v\in L_n$ there exists two letters $a,b$ such that $av$ and $vb$ are elements of $L_{n+1}$, and all factors of length $n$ of elements of $L_{n+1}$ are in $L_n$.
\end{definition}

\begin{definition}\label{defcomp}
The complexity function of a language $L$ is the function $p:\mathbb{N}\rightarrow\mathbb{N}$ defined by $p(n)=card(L_n)$.
\end{definition}

\subsubsection{Substitutions}
{\bf Substitution:}
A substitution $\sigma$ is an application from an alphabet $\mathcal{A}$ to the set $\mathcal{A}^*\setminus\{\varepsilon\}$ of nonempty finite words on $\mathcal{A}$. It extends to a morphism of $\mathcal{A}^*$ by concatenation, that is $\sigma(uv)=\sigma(u)\sigma(v)$. 

{\bf Fixed point}: A fixed point of a substitution $\sigma$ is an infinite word $u$ with $\sigma(u)=u$. 

{\bf Periodic point}: A periodic point of a substitution $\sigma$ is an infinite word $u$ with $\sigma^k(u)=u$ for some $k>0$.

Let $\sigma$ be a substitution over the alphabet $\mathcal{A}$, and $a$ be a letter such that $\sigma(a)$ begins with $a$ and $|\sigma(a)|\geq 2$. Then there exists a unique fixed point $u$ of $\sigma$ beginning with $a$. This infinite word is the limit of the sequence of finite words $\sigma^n(a)$.
\subsubsection{Complexity}

 First we recall a result of the second author concerning combinatorics of words \cite{Ca}.
\begin{definition}
Let $L=(L_n)_{n\in\mathbb{N}}$ be a language. For any
$n\geq 0$ let $s(n)\!:=p(n+1)-p(n)$. For $v \in L_n$ let

$$m_{l}(v)= card\{a\in \mathcal{A}, av\in L_{n+1}\},$$
$$m_{r}(v)= card\{b\in \mathcal{A},vb\in L_{n+1}\},$$
$$m_{b}(v)= card\{(a,b)\in \mathcal{A}^2, avb\in L_{n+2}\}.$$
\begin{itemize}
\item A word $v$ is called right special if $m_{r}(v)\geq 2$. 
\item A word $v$ is called left special if $m_{l}(v)\geq 2$. 
\item A word $v$ is called bispecial if it is right and left special.
\item $BL_n$ denotes the set of bispecial words of length $n$. 
\item $b(n)$ denotes the sum $b(n)=\displaystyle\sum_{v\in BL_n}{i(v)}$,
where $i(v)=m_{b}(v)-m_{r}(v)-m_{l}(v)+1$.
\end{itemize}
\end{definition}

\begin{lemma}\label{julien}
Let $L$ be a language. Then the complexity of $L$ satisfies for every integer $n\geq 0$:
$$s(n+1)-s(n)=b(n).$$
\end{lemma}

For the proof of the lemma we refer to \cite{Ca} or
\cite{Ca.Hu.Tr}.

\begin{definition}
A word $v$ such that $i(v)< 0$ is called a weak bispecial. 
A word $v$ such that $i(v)>0$ is called a strong bispecial.
A bispecial word $v$ such that $i(v)=0$ is called a neutral bispecial.
\end{definition}

\begin{definition}
Let $v$ be a finite word. The infinite periodic word with period $v$ will be denoted $v^\omega$.
\end{definition}

\subsection{Coding}\label{coddual}
We introduce a coding for the dual billiard map. Recall that the polygon $P$ has $k$ vertices. 
\begin{definition}
The sides of $P$ can be extended in half-lines in the following way: Denote these half lines by $(d_i)_{0\leq i\leq k-1}$ we assume that the angle $(d_i,d_{i+1})$ is positive. They form a cellular decomposition of $\overline{P}$ into $k$ closed cones $V_0,\dots, V_{k-1}$. By convention we assume that the half line $d_i$ is between $V_{i-1}$ and $V_i$ see Figure \ref{defcones}.
\end{definition}
\begin{figure}[h]
\begin{center}
\includegraphics[width= 5cm]{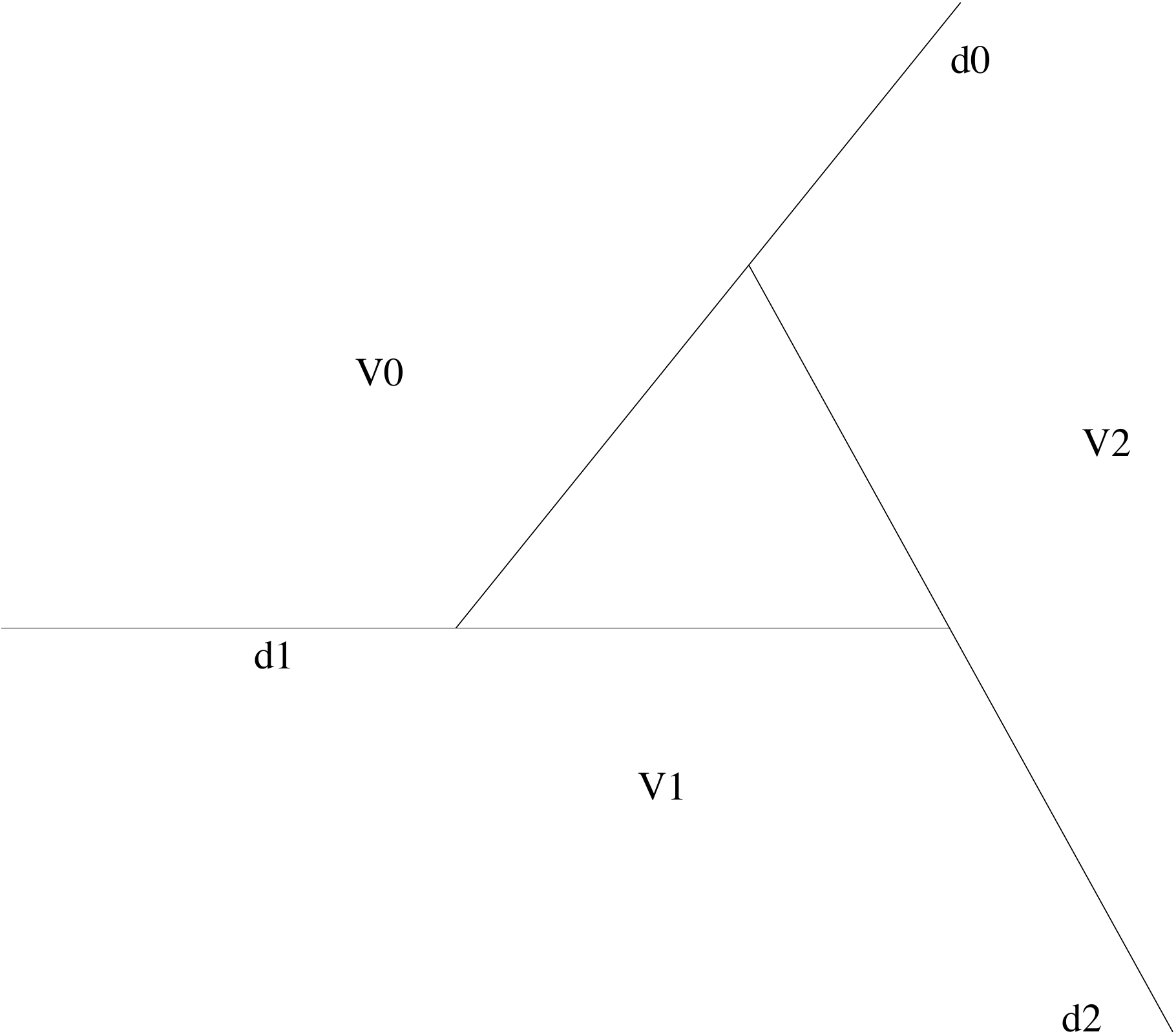}
\caption{Partition}\label{defcones}
\end{center}
\end{figure}

Now we define the coding map, we refer to Definition \ref{defdual}:
\begin{definition}\label{defcoddual}
Let $\Phi$ be the map
$$\begin{array}{ccccc}
\rho&:& \overline{P}\setminus Y &\to& \{0,\dots ,k-1\}^\mathbb{N}\\ 
&&M&\mapsto& (u_n)_{n\in\mathbb{N}}\end{array}$$
where $u_n=i$ if and only if $T^nM\in V_i.$
\end{definition}

Now consider the factors of length $n$ of $u$, and denote this set by $L_n(M)$. Remark that $T^nM\in V_i\cap V_{i+1}$ is impossible if $M\notin Y$.

\begin{definition}\label{lang-dyn}
We introduce the following set
$$L_n=\bigcup_{M\in\overline{P}\setminus Y} L_n(M).$$
This set corresponds to all the words of length $n$ which code outer billiard orbits. Then the set $L=\bigcup_nL_n$ is a language. It is the language of the outer billiard map.
We denote the complexity of $L$ by $p(n)$, see Definition \ref{defcomp}.
$$p(n)=card (L_n).$$
\end{definition}

\begin{definition}
The set $\{0,\dots ,k-1\}^\mathbb{N}$ has a natural product topology. Then $X$ denotes the closed set $X=\overline{\rho(\overline{P}\setminus Y)}$. 
\end{definition}
The link between $X$ and $L$ is the same as in Remark \ref{remcombinmot} between $X_u$ and $L_u$.

\section{Simplification of the problem }\label{secsimplif}
\subsection{First remarks}
\begin{lemma}\label{lem-isom-affin}
Let $P$ be a convex polygon, and $h$ an affine map of $\mathbb{R}^2$ preserving the orientation, then the languages of the outer billiard maps outside $P$ and $h(P)$ are the same.
\end{lemma}
\begin{proof}
The proof is left to the reader.
\end{proof}

Remark that an affine map preserves the set of lattices. Thus if $P$ is rational, then $h(P)$ is rational for each affine map $h$. Also the outer billiard map outside any triangle has the same language, so it is sufficiant to study the equilateral triangle, see Lemma \ref{lem-isom-affin}.

\subsection{A new coding for the regular polygon}\label{newcoding}
Let $P$ be a regular polygon with $k$ vertices, and $R$ be the rotation of angle $-2\pi/k$, centered at the center of the polygon. Consider one sector $V_0$ and define the map:
$$\begin{array}{cccc}
& \displaystyle\bigcup_iV_i   & \to& \mathbb{N}   \\
 & y & \mapsto & n_y 
\end{array}$$
where the integer $n_y$ is the smallest integer which maps the sector $V_i$ containing $y$ to $V_0$.
Then we define a new map
\begin{definition}\label{Tcone}
The map $\hat{T}$ is defined in $V_0$ by the formula
$$\hat{T}(x)=R^{n_{Tx}}Tx.$$ 
\end{definition}

\begin{lemma}\label{isom-newcoding}
The integer $n_{Tx}$ takes the values $1$ to $j=\lfloor \frac{k+1}{2}\rfloor$. The map 
$\hat{T}$ is a piecewise isometry on $j$ pieces.
\end{lemma}
\begin{proof}
We will treat the case where $k$ is even, the other case is similar. Assume $k=2k'$, then we can assume that the regular polygon has as vertices the complex numbers $e^{i\pi n/k'}, n=0\dots k-1$ and that $V_0$ has $1$ as vertex. Consider the cone $V=TV_0$ of vertex $1$ obtained by central symmetry from the cone $V_0$. We must count the number of intersection points of this cone with the cones $V_i, i=0\dots k-1.$ The polygon is invariant by a central symmetry, this symmetry maps the cone $V_i$, to the cone $V_{k'+i}$. We deduce that if the cone $V_i$ intersects $V$, then the cone $V_{k'-i}$ cannot intersect it. Moreover it is clear that each cone $V_i, 1\leq i\leq k'$ intersects $V$, thus $n_{T_x}$ takes the values $1$ to $k'$. Now for each value of $n_{Tx}$ we obtain an isometry, thus we obtain a piecewise map defined on $j=k'$ sets.
\end{proof}

\begin{figure}
\begin{center}
\includegraphics[width=5cm]{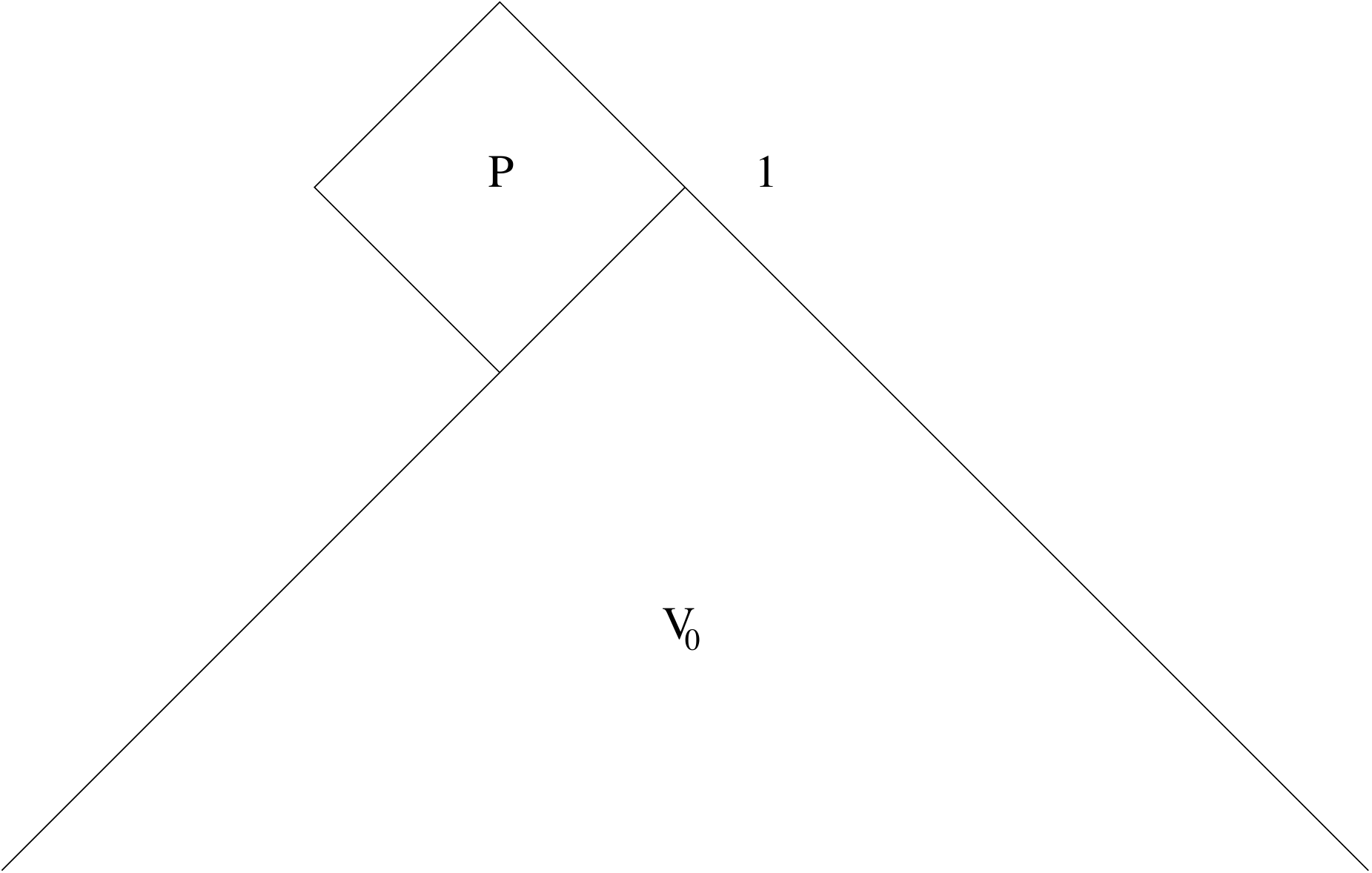}
\caption{New coding}
\end{center}
\end{figure}

\begin{definition}\label{def-lang-hat}
Let $\eta$ be the map defined as
$$\begin{array}{ccccc}
\eta & : &V_0\setminus Y  & \to & \{1,\dots, j\}^\mathbb{N} \\
& & x & \mapsto & (n_{T(\hat{T}^{i}x)})_{i\in\mathbb{N}} \\
\end{array}$$

As in Definition \ref{lang-dyn}, let $L'$ be the language of $\hat{T}$ related to the coding $\eta$.
\end{definition}

\begin{lemma}
We have 
$$\hat{T}^k(x)=R^{n_{T^kx}}T^kx.$$
\end{lemma}
\begin{proof}
This lemma is a consequence of the following fact:
If $A,B$ are two consecutive vertices of the polygon for the orientation, denote by $s_A,s_B$ the central symmetries about them, then we have $Rs_B=s_AR.$ This relation implies that $R$ and $T$ commute: 
$$RT=TR.$$
\end{proof}

\begin{lemma}\label{lem-mod}
If $x\in V_0\setminus Y$, then the codings $(u_n)=\rho(x)$ and $(v_n)=\eta(x)$ are linked by $$v_n=u_{n+1}-u_n \quad mod \quad k.$$  
\end{lemma}
\begin{proof}
Consider two consecutive elements of the sequence $u$ with values $a,b$. It means $T^nx\in V_a, T^{n+1}x\in V_b$. Now the rotation $R^{b-a}$ maps $V_b$ to $V_a$, thus we deduce that $R^{b-a}T[T^{n}x]$ belongs to $V_a$, this implies with the help of preceding Lemma that $v_n=b-a\quad mod \quad k.$
\end{proof}

The preceding Lemma implies that the study of the map $\hat{T}$ will give information for $T$. 

\begin{lemma}\label{changealphabet}
With the notations of Definitions \ref{lang-dyn}, \ref{def-lang-hat} we have
$$p_L(n)=kp_{L'}(n-1).$$
Moreover the map $\begin{array}{ccc}L &\to& L'\\ u &\mapsto & v\end{array}$ defined by $v_i=u_{i+1}-u_i$, for $0\leq i\leq n-2$ if $u=u_0u_1\dots u_{n-1}$ is a $k$-to-one map.
\end{lemma}
\begin{proof}
The regular polygon is invariant by the rotation $R$, thus the points $x$ and $R^ix, 0\leq i\leq k-1$ have the same coding in $L'$.  Thus the map is not injective and the pre-image of a word consists of $k$ word. By definition it is surjective. Now the formula for the complexity function is an obvious consequence of the formula $v_n=u_{n+1}-u_n$, see Lemma \ref{lem-mod}.
\end{proof}


\section{Background}\label{secback}
Few results are known about the complexity of the outer billiard map.
We recall them in this section.

First of all Buzzi proved in \cite{Buz.01} that the topological entropy is zero.
\begin{theorem}
For any piecewise isometry $T$ of $\mathbb{R}^n$ we have $h_{top}(T)=0$. As a consequence we have $\displaystyle\lim_{+\infty}\frac{\log{p(n)}}{n}=0$.
\end{theorem}

Then Gutkin and Tabachnikov proved in \cite{Gu.Ta.06} the following result.
\begin{theorem} Let $P$ be a convex polygon 
\begin{itemize}
\item If $P$ is a regular polygon with $k$ vertices then there exist $a,b>0$ such that
$$an\leq p(n)\leq  bn^{r+2}.$$ The integer $r$ is the rank of the abelian group generated by translations in the sides of $P$. We have $r=\phi(k)$, where $\phi$ is the Euler function.

\item If $P$ is a rational polygon, then there exist $a,b>0$ such that
$$an^2\leq p(n)\leq bn^2.$$
\end{itemize}
\end{theorem}
In fact their theorem concerns the more general family of quasi-rational polygons that includes  the regular ones, but we will not prove a result about this family of polygons. Remark that for the regular pentagon, we have $r=4$.


\section{Results}\label{secresult}
We obtain two types of results: The description of the language of the dynamics, and the computation of the complexity. 
The results are obtained for two types of polygons: the triangle, the square and the regular hexagon which are rational polygons; and the regular pentagon which is a quasi-rational polygon. 
\subsection{Languages}
We characterize the languages of the outer billiard map outside regular polygons:
We will use Lemma \ref{changealphabet} and work with the language $L'$. Moreover we will give only infinite words. The language is the set of finite words which appear in  the infinite words.
\begin{definition}
Consider the three following endomorphisms of the free group $F_3$ defined on the alphabet $\{1;2;3\}$
$$\sigma: \begin{cases}1\rightarrow 1121211 \\ 2\rightarrow 111\\ 3\rightarrow 3\end{cases} \psi:\begin{cases}1\rightarrow 2232232\\ 2\rightarrow 232\\ 3\rightarrow 2^{-1} \end{cases} \xi:\begin{cases}1\rightarrow 23222\\ 2\rightarrow 2\\ 3\rightarrow 3\end{cases}$$ 

\end{definition}

\begin{theorem}\label{langages}
Let $P$ be a triangle, a square, a regular hexagon or a regular pentagon. Then the language $L'$ of the dynamics of $\hat{T}$ is the set of factors of the periodic words of the form $z^\omega$ for z$\in Z$, where
\begin{itemize}
\item If $P$ is a triangle 

$$Z=\bigcup_{n\in\mathbb{N}} \{1(21)^n, 1(21)^n1(21)^{n+1}\}.$$

\item If $P$ is the square 
$$Z=\bigcup_{n\in\mathbb{N}}\{12^n\}.$$ 

\item If $P$ is the regular hexagon
$$Z=\bigcup_{n\in\mathbb{N}}\{23^n, 23^n23^{n+1}\}\cup \{1\}.$$

\item If $P$ is the regular pentagon then $Z$ is the union of
$$\bigcup_{n\in\mathbb{N}}\{\sigma^n(1), \sigma^n(12)\},$$
$$\bigcup_{n,m\in\mathbb{N}} \{\psi^m(2), \psi^m(2223), \psi^m\circ\sigma^n(1), \psi^m\circ\sigma^n(12)\},$$
$$\bigcup_{n,m\in\mathbb{N}} \{\psi^m\circ\xi\circ\sigma^n(1), \psi^m\circ\xi\circ\sigma^n(12)\}.$$
\end{itemize}
\end{theorem}

\subsection{Complexity}
In the statement of Theorem \ref{calccomp} we give the formula for $p_{L'}$ . Lemma \ref{changealphabet} can be used to obtain the formula for the complexity of the language $L$.

\begin{theorem}\label{calccomp}

\begin{itemize}
\item For a triangle, we have
$$p_{L'}(n)=\frac{5n^2+14n+f(r)}{24},$$
where $r=n\quad mod\quad12$ and $f(r)$ is given by

\begin{tabular}{|c|c|c|c|c|c|c|c|c|c|c|c|c|c|c|}
\hline
r&0&1&2&3&4&5&6&7&8&9&10&11\\
\hline
f(r)&24&29&24&9&8&21&24&17&0&-3&8&9\\
\hline
\end{tabular}

\item For a square we obtain:
 $$p_{L'}(n)=\frac{1}{2}\lfloor\frac{(n+2)^2}{2}\rfloor.$$
 
\item For a regular hexagon:
$$p_{L'}(n)=\lfloor \frac{5n^2+16n+15}{12}\rfloor.$$

\item For a regular pentagon, let $\beta$ be the real number:
\begin{align*}
\beta=&\frac{14}{15}+
\displaystyle\sum_{n\geq 0}(\frac{7}{48.6^n.2+14+2(-1)^n}+\frac{7}{18.6^n.2+14-(-1)^n})\\
&-\displaystyle\sum_{n\geq 0}(\frac{7}{78.6^n.2+14+5(-1)^n}+\frac{7}{48.6^n.2+14-5(-1)^n}).
\end{align*}
then we have
$$p_{L'}(n)\sim \frac{\beta n^2}{2},$$
$$\beta\sim 1.06$$
\end{itemize}
\end{theorem}


\section{Induction}\label{secinduc}
In this section we will consider the map $\hat{T}$ of Definition \ref{Tcone}, and consider its first return map on different sets. Since we consider the same map for three different polygons, we will denote it by $\hat{T}_{square},\hat{T}_{tria},\hat{T}_{hexa}$ in the different cases.

\subsection{Induction and substitution}
In this subsection we recall some usual facts about the coding of a map and the coding of the map obtained by induction of the initial map.\\
Let $X$ be an open subset of $\mathbb{R}^d$, and $T$ a map from $X$ to $X$. Assume there is a partition of $X$ by the sets $U_i, i=0\dots k-1$, then consider the language $L_U$ obtained by the associated coding map as in Definition \ref{defcoddual}. Now assume that the first return map of $T$ on $U_1$ is defined and denote it by $T_{U_1}$. $$T_{U_1}(x)=T^{n_x}(x),$$ where $n_x$ is the smallest integer such that $T^nx$ belongs to $U_1$.
Consider the partition of $U_1$ in which the return time $n_{x}$ is constant on each part of the partition. This partition is associated to a coding map for the map $T_{U_1}$. We denote the associated language by $L_{U_1}$.
We will give conditions to describe the language of this first return map 
in terms of the initial language.\\

If $x$ is in one cell of the partition of $U_1$ then denote by $v_i$ the finite word  of length $n_x$which codes the orbit $T^jx, j=0\dots n_x-1$ where $n_x$ is the return time of $x$ in $U_1$.
Define the substitution $\alpha$ by 
$$\begin{cases}\alpha_{U_1}(1)=v_1\\ \vdots\\ \alpha_{U_1}(k)=v_k
\end{cases}$$

\begin{lemma}\label{leminducsub}
\begin{itemize}
\item In general we have $$\alpha_{U_1}(L_{U_1})\subset L_U.$$
\item If the maps are defined for every set $U_i$, we have $\displaystyle\bigcup_{i}\alpha_{U_i}(L_{U_i})=L_U.$ 
\item Assume there exists an homeomorphism $h$ between $X$ and $U_1$ such that for all $x\in X$ we have $h^{-1}\circ T_{U_1}\circ h(x)=T(x)$. Then we have 
$$L_{U_1}=L_U.$$
Moreover $L_U$ is stable by $\alpha_{U_1}$.
\end{itemize}
\end{lemma}
The proof is classical and can be found in the litterature.

\subsection{Square}
\subsubsection{Computation}

\begin{lemma}
The map $\hat{T}_{square}$ is a piecewise isometry on two sets, see Figure \ref{figcarre}.
On the first set $U_1$ it is a rotation of angle $\pi/2$. On the second set $U_2$ it is a translation. The center of the rotation is on the bissector of the sector.
\end{lemma}
\begin{proof}
The proof is left to the reader.
\end{proof}

\begin{lemma}
Consider the induction of $\hat{T}_{square}$ on the set $U_2$. The induction will be denoted $\hat{T}_{square,2}$.
This map is related to $\hat{T}_{square}$ by
$$\hat{T}_{square}\circ t_2=t_2\circ \hat{T}_{square,2}$$
where $t_2$ is the vertical translation which maps $V_0$ to $U_2$ see Figure \ref{figcarre}.
\end{lemma}
\begin{proof}
To prove this result, we consider the image of $U_2$ by $\hat{T}_{square}$, it  splits in two parts: one which is in $U_2$ and one which comes back after one iteration. The splitting corresponds to the new partition of $U_2$.
\end{proof}

\begin{corollary}
We deduce that the language is invariant under the substitution 
$\alpha_{car}:\begin{cases}1\rightarrow 12\\ 2\rightarrow 2\end{cases}$
\end{corollary}
\begin{proof}
This result is an obvious consequence of Lemma \ref{leminducsub}. Indeed there are two different return times which are equal to $2$ and $1$ by preceding Lemma.  
\end{proof}
\subsubsection{Proof of Theorem \ref{langages} for the square}
Since $T_{U_1}$ is a rotation of angle $\pi/2$, there a symmetry of order four, thus $(1^4)^\omega=1^\omega$ is a periodic word for $\hat{T}_{square}$. Now by preceding Lemma the  words $\alpha_{car}^n(1)^\omega$ are words of the language for every integer $n$. Now we have $\alpha_{car}^n(1)=12^{n-1}$ and the stability of the language by $\alpha_{car}$ finishes the proof.

\begin{figure}[h]
\begin{center}
\includegraphics[width= 6cm]{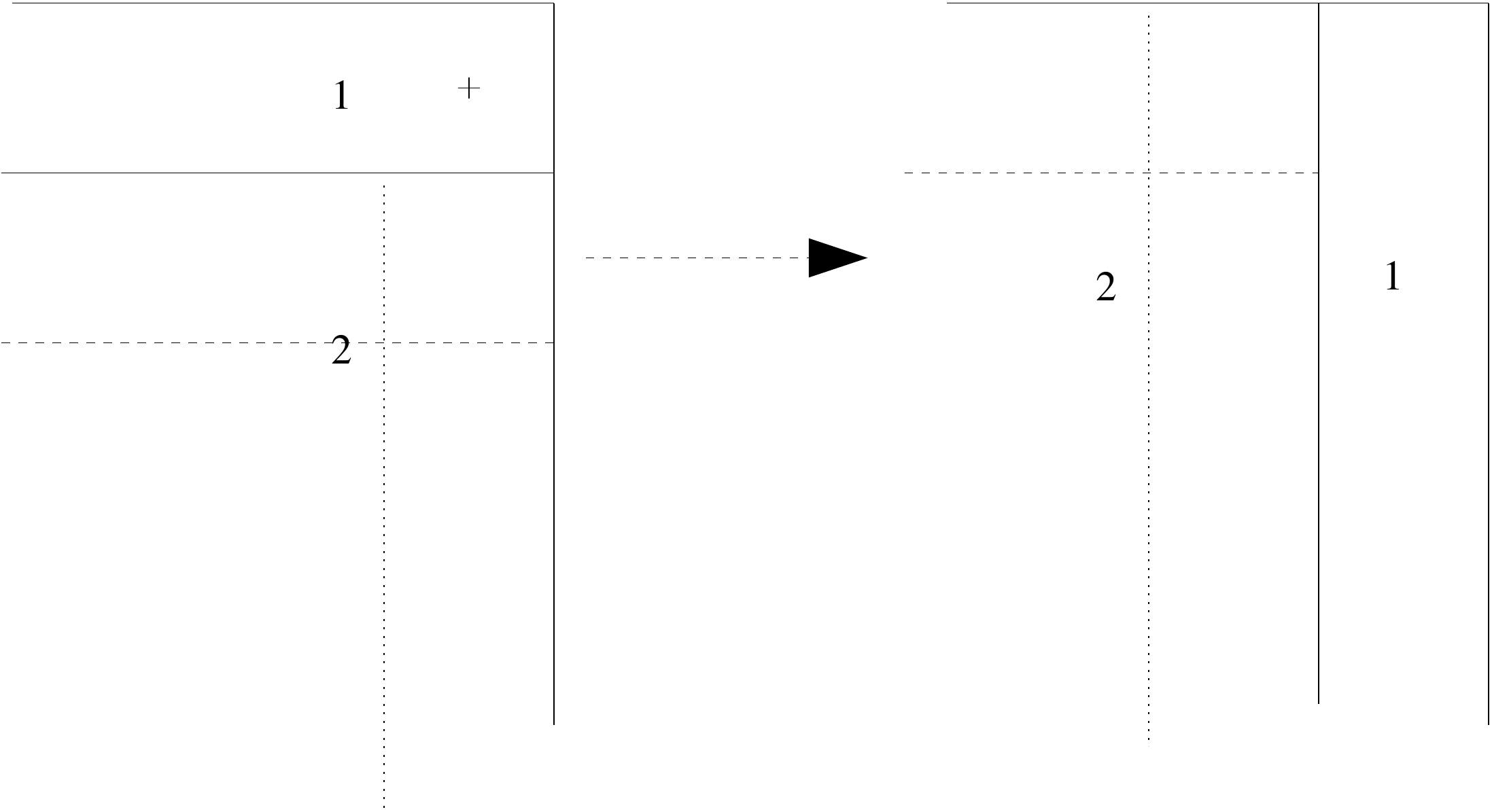}
\caption{Dynamics of $\hat{T}_{square}$}
\label{figcarre}
\end{center}
\end{figure}


\subsection{Regular hexagon}
\subsubsection{Computation}
In this part we state some results similar to the ones of preceding section.
\begin{lemma}\label{hexa-un}
The map $\hat{T}_{hexa}$ is a piecewise isometry on three sets.
On the first set it is a rotation of angle $2\pi/3$. On the second set it is a rotation of angle $\pi/3$, see Figure \ref{hatThexa} .On the third set it is a translation. The first set is invariant by $\hat{T}_{hexa}$.
\end{lemma}

\begin{figure}[h]
\begin{center}
\includegraphics[width= 6cm]{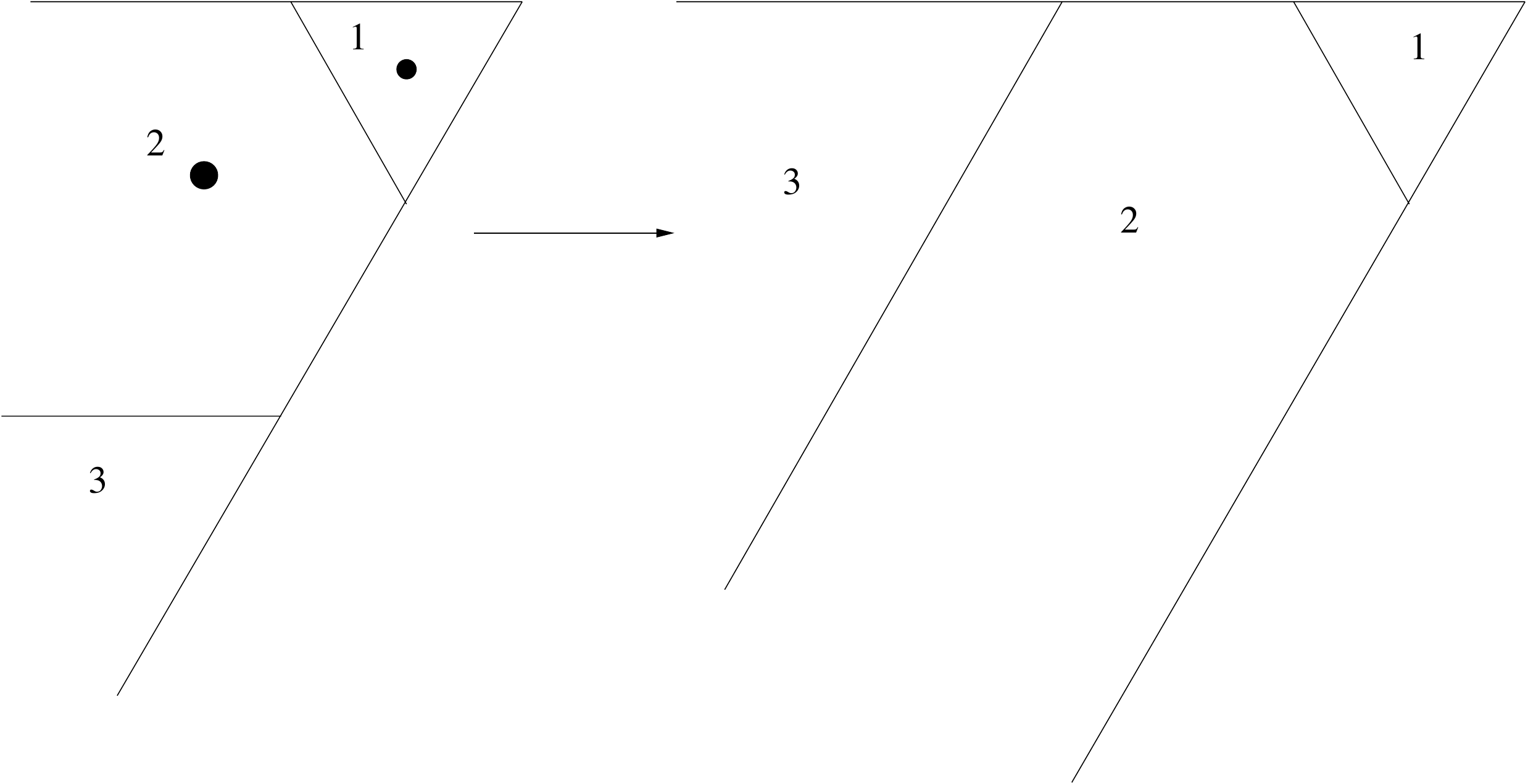}
\caption{Return map $\hat{T}_{hexa}$}\label{hatThexa}
\end{center}
\end{figure}

\begin{lemma}
Consider the first return map of $\hat{T}_{hexa}$ on the set $U_2$. The induction will be denoted $\hat{T}_{hexa,2}$.
Then there exists a translations $t'$ such that
$$\hat{T}_{hexa,2}\circ t'=t'\circ \hat{T}_{hexa}.$$
\end{lemma}

The proofs of these results are left to the reader, you can use the expressions of the rotations with complex numbers.

\begin{corollary}\label{hexa-deux}
We deduce that the language is invariant under the substitution 
$\alpha_{hex}=\begin{cases}2\rightarrow 23\\ 3\rightarrow 3\end{cases}$
\end{corollary}
\begin{proof}
The proof is a consequence of Lemma \ref{leminducsub}.
\end{proof}

\subsubsection{Proof of Theorem \ref{langages} for the hexagon}
The word $1^\omega$ is clearly a periodic word of the language, moreover the word $1^\omega$ is in the subshift by Lemma \ref{hexa-un} and no other word in $L'$ contains $1$. Now we can work on the sublanguage defined on the alphabet $\{2;3\}$. By Corollary \ref{hexa-deux} we deduce that the substitution $\alpha_{hex}$ is an invariant for the words of the language associated to $\hat{T}_{hexa}$. Since the words $2^\omega$ and $(223)^\omega$ are clearly periodic words, we deduce the result: The periodic words of the language are
$$\alpha_{hex}^n(2)^\omega, \alpha_{hex}^n(223)^\omega.$$

\subsection{Triangle}
\subsubsection{Computation}
\begin{lemma}
The map $\hat{T}_{tria}$ is a piecewise isometry on two sets, see Figure \ref{hatTtriang}.
On the first set it is a rotation of angle $\pi/3$. On the second set it is a rotation of angle $-\pi/3$. Both rotations have their centers on the bissector of the sector.
\end{lemma}
\begin{proof}
The proof is left to the reader.
\end{proof}

\begin{figure}[ht]
\begin{center}
\includegraphics[width= 6cm]{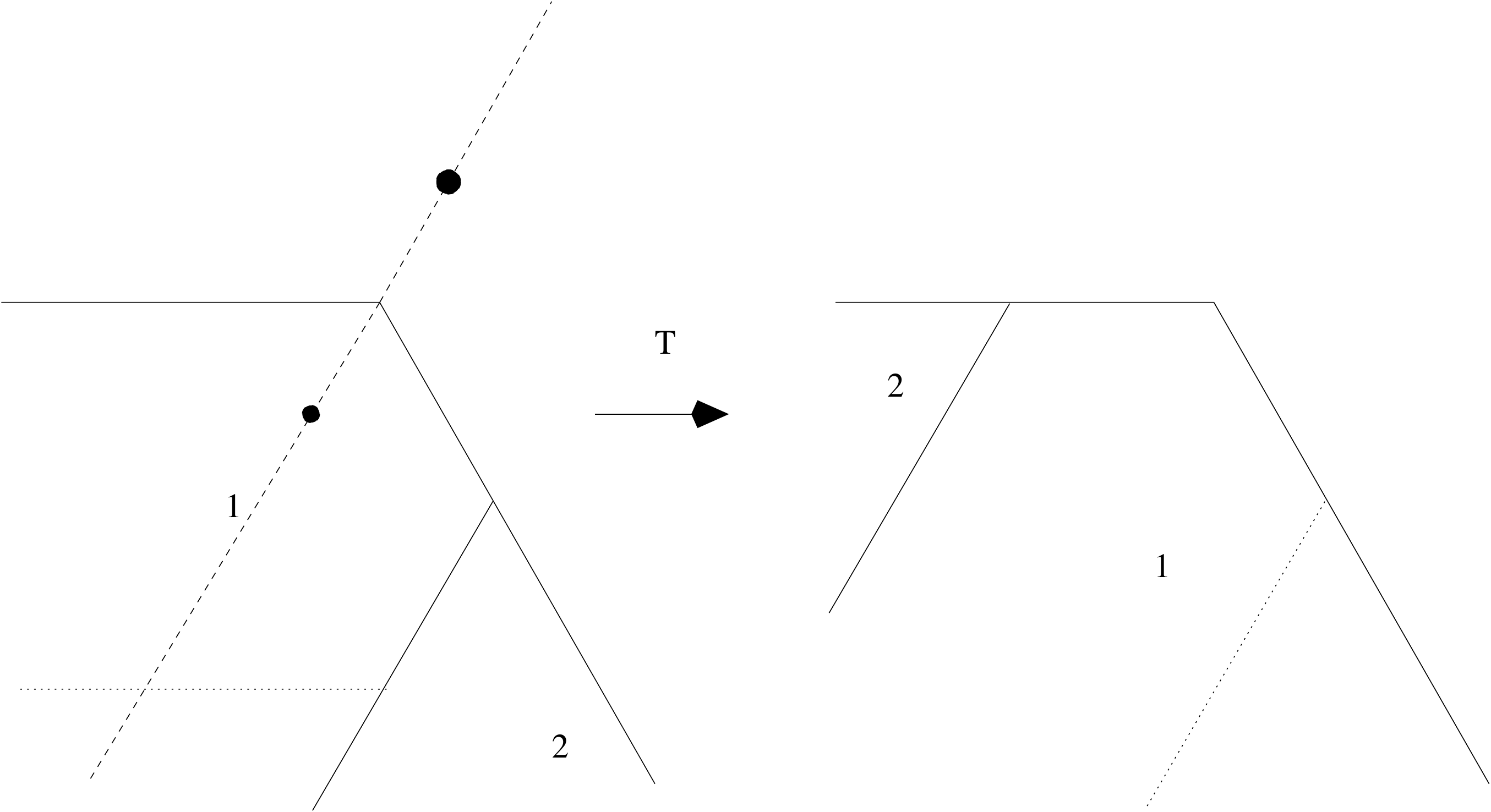}
\caption{Dynamics of $\hat{T}_{tria}$}\label{hatTtriang}
\end{center}
\end{figure}
\begin{lemma}
The induction map of $\hat{T}_{tria}$ on the set coded by $2$ is a piecewise isometry denoted by $\hat{T}_{tria,2}$ such that there exists a translation $u$
$$\hat{T}_{hexa}\circ u=u\circ \hat{T}_{tria;2}.$$ 
\end{lemma}
\begin{proof}
The proof is left to the reader.
\end{proof}

\begin{corollary}\label{hexa-triang}
 There is a bijective map between the language of the outer billiard map out of the regular hexagon and out of the triangle which is 
  given by
$$\begin{array}{ccccc}
& &\displaystyle L'_{hexa}\  & \to L'_{tria} &  \\
& & \begin{cases}1\\2\\3\\ \end{cases} & \mapsto & \begin{cases}2111\\ 211\\ 21\end{cases} \\
\end{array}$$
\end{corollary}
\begin{proof}
The returns words for $2$ in the induction are the followings
$$21;211;2111$$ then if the outer billiard outside the regular hexagon is coded by $a,b,c$ we deduce that $a=2111,b=211,c=21$. This means that there is a bijection between the two languages by  Lemma \ref{leminducsub}.
\end{proof}
 
\begin{corollary}
We use the result proved in the preceding section, and we deduce
that the language is invariant by the substitution
$\alpha_{tria}:\begin{cases}1\rightarrow 121\\ 2\rightarrow 1^{-1}\end{cases}$.
\end{corollary}
\begin{proof}
There is a bijection between the two languages, then we use Lemma \ref{hexa-deux} which gives the invariance of the language by a substitution. It remains to write this susbtitution in the new alphabet.
\end{proof}

\subsubsection{Proof of Theorem \ref{langages} for the triangle}
 The proof is a consequence of the preceding corollary.
 
With the preceding result we deduce that the language is given by 
$$\bigcup_{n\in\mathbb{N}}\alpha_{tria}^n(1)^\omega\cup \alpha_{tria}^n(1121)^\omega.$$

\begin{figure}[h]
\begin{center}
\includegraphics[width= 6cm]{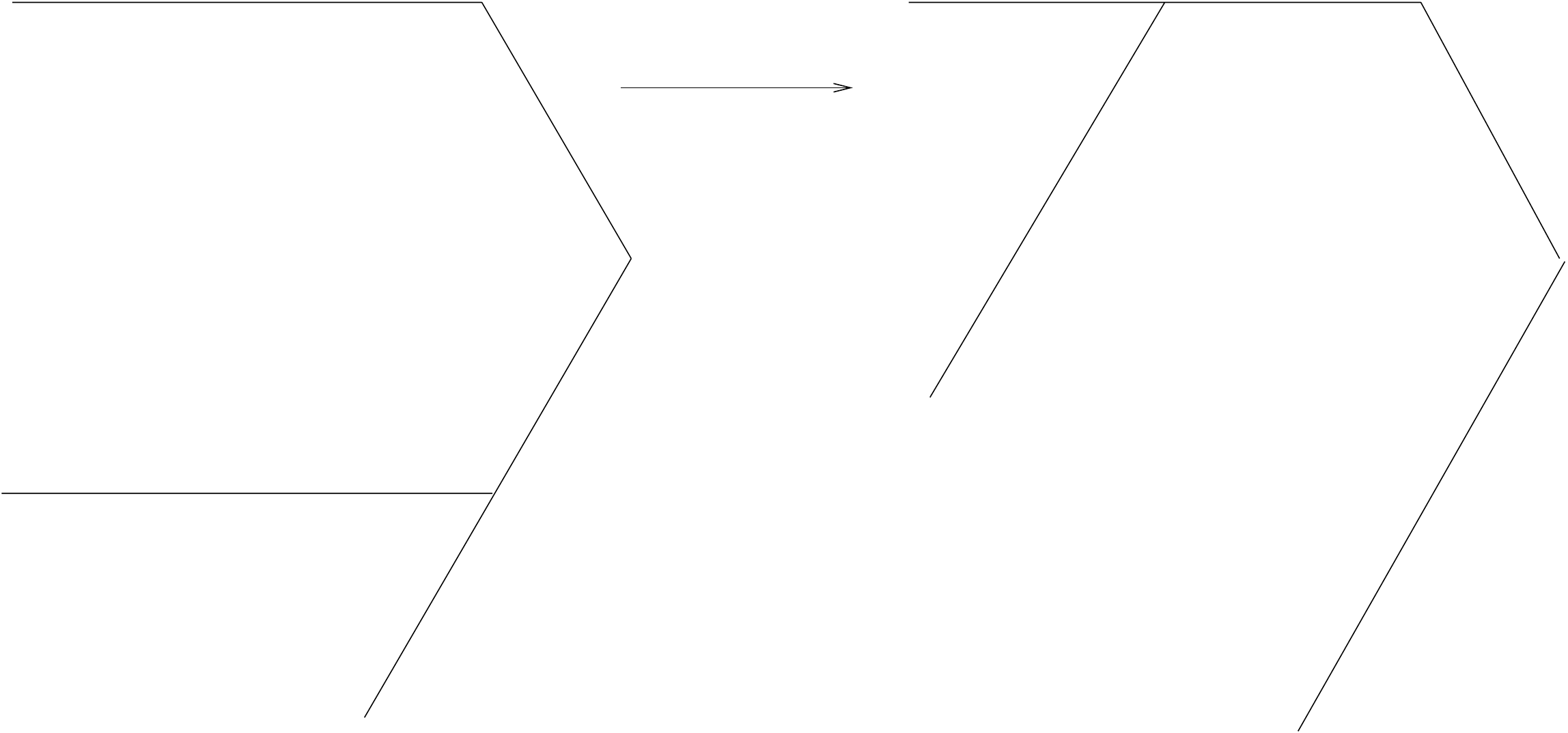}
\caption{Induction map for the triangle}
\end{center}
\end{figure}


\section{Computation of the complexity}\label{seccalcomp}
In this section we use Theorem \ref{langages} to compute the bispecial words of the different languages. It will allow us to obtain the complexity function in the different cases.

\subsection{Square}
\subsubsection{Bispecial words}
\begin{lemma}
If $w$ is a bispecial word, then we have $w=s\alpha_{car}(v)p$ with $s=2, p=\varepsilon$. Moreover $v$ is a bispecial word with the same extensions as $w$, 
\end{lemma}
\begin{proof}
First we remark the following fact: if we omit the word $1^\omega$, then the letter $1$ is isolated. Thus a bispecial word, of length bigger than $2$, begins and ends by $2$. Now let $w$ be a bispecial word of this language, then there exists an integer $n$ such that $w=2^nv'$. Now $v'$ begins by $1$, thus we can write $w=2^n\alpha(w)=2\alpha(v)$. Moreover $w$ is a bispecial word of length one.
\end{proof}

We deduce the following result.
\begin{corollary}\label{bisp-car}
The bispecials words of the language $L'$ of the square are the words $2^n, n\in\mathbb{N}$, and $2^n12^n, n\in\mathbb{N}$.
\end{corollary}
\begin{proof}
We use preceding Lemma and the fact that $1,2$ are bispecial words. Thus all the bispecial words are of the form 
$$\hat{\alpha_{car}}^n(2), \hat{\alpha_{car}}^n(1), n\in\mathbb{N}$$
where $\hat{\alpha_{car}}(v)=2\alpha_{car}(v)$.
\end{proof}

\subsubsection{Proof of Theorem \ref{calccomp} for the square}
First we begin by a corollary of Corollary \ref{bisp-car}.
\begin{corollary}\label{cal-bisp-car}
We deduce that $b(n+1)+b(n)=1 \quad \forall n\in\mathbb{N}$.
\end{corollary}
\begin{proof}
The words $2^n$ can be extended into the words $$12^n1,12^n2,22^n2,22^n1.$$ We deduce that 
$i(2^n)=1.$ The words $2^n12^n$ can be extended in the words $12^n12^n1,22^n12^n2$, thus we have $i(2^n12^n)=-1$. We deduce that $b(n)$ equals $0$ if $n$ is odd and $1$ otherwise. The result follows from Lemma \ref{julien}.
\end{proof}

We have $p(0)=1, p(1)=2, p(2)=4$, we deduce $s(0)=1, s(1)=2$. 
Then Corollary \ref{cal-bisp-car} shows that 
$$s(n+2)-s(n)=1 \quad n\geq 1 .$$
$$s(2n)=n+1\quad n\geq 0 .$$
$$s(2n+1)=n+2\quad n\geq 0.$$
$$p(n)=\displaystyle\sum_{i=1}^{n-1}s(i)+p(1).$$
Thus we deduce for $n\geq 1$:
$$p(2n)=\displaystyle\sum_{2\leq 2i\leq 2n-1}(i+1)+\displaystyle\sum_{1\leq 2i+1\leq 2n-1} (i+2)+p(1).$$
$$p(2n)=\displaystyle\sum_{1\leq i\leq n-1}(i+1)+ \displaystyle\sum_{0\leq i\leq n-1}(i+2)+p(1).$$
$$p(2n)=n(n-1)/2+n-1+n(n-1)/2+2n+2.$$
$$p(2n)=n^2+2n+1.$$
$$p(2n+1)=\displaystyle\sum_{2\leq 2i\leq 2n}(i+1)+\displaystyle\sum_{1\leq 2i+1\leq 2n-1} (i+2)+p(1).$$
$$p(2n+1)=\displaystyle\sum_{1\leq i\leq n}(i+1)+\displaystyle\sum_{0\leq i\leq n-1} (i+2)+p(1).$$
$$p(2n+1)=n(n+1)/2+n+n(n-1)/2+2n+2.$$
$$p(2n+1)=n^2+3n+2.$$
We observe that $p(n)=\frac{1}{2}\lfloor\frac{(n+2)^2}{2} \rfloor.$

\subsection{Regular hexagon}
\subsubsection{Bispecial words}
\begin{lemma}\label{bisp-hexa}
Define a new map by $\hat{\alpha_{hexa}}(v)=3\alpha_{hexa}(v)$.
The bispecial words of this language are 
\begin{itemize}
\item $\hat{\alpha_{hexa}}^n(\varepsilon)=3^n, n\in\mathbb{N}$.
\item $\hat{\alpha_{hexa}}^n(2)=3^n23^{n}, n\in\mathbb{N}$.
\item $\hat{\alpha_{hexa}}^n(23)=3^n23^{n+1}, n\in\mathbb{N}.$
\item $\hat{\alpha_{hexa}}^n(32)=3^{n+1}23^n, n\in\mathbb{N}.$
\item $\hat{\alpha_{hexa}}^n(232)=3^{n}23^{n+1}23^n, n\in\mathbb{N}.$
\item $\hat{\alpha_{hexa}}^n(22)=3^n23^n23^n, n\geq 1$ and $22$ if $n=0$.
\end{itemize}
\end{lemma}
\begin{proof}
We use the same method as in the case of the square. Indeed the two susbtitutions $\alpha_{car}, \alpha_{hexa}$ are similar, the proof ends with the remark that $\varepsilon,2,23,32,22,232$ are bispecial words, thus all the bispecial words are iterations of one of these six words under $\hat{\alpha_{hexa}}$.
We deduce the six families of bispecial words
$$\hat{\alpha_{hexa}}^n(\varepsilon), \hat{\alpha_{hexa}}^n(2), \hat{\alpha_{hexa}}^n(23), \hat{\alpha_{hexa}}^n(32), \hat{\alpha_{hexa}}^n(22), \hat{\alpha_{hexa}}^n(232), n\in\mathbb{N}.$$

\end{proof}

\subsubsection{Proof of Theorem \ref{calccomp} for the regular hexagon}
\begin{proposition}\label{calculhexa}
The complexity of the language $L'$ associated to the regular hexagon satisfies for all integer $n$:
$$\displaystyle\sum_{i=n}^{11+n}b(i)=10.$$
\end{proposition}
\begin{proof}
By Lemma \ref{bisp-hexa} the bispecial words form six famillies of words. 
The words $\hat{\alpha_{hexa}}^n(232), \hat{\alpha_{hexa}}^n(22)$ are neutral bispecial words, thus they do not count in the computation. 
For the other words we have 

$$i(\hat{\alpha_{hexa}}^n(\varepsilon))=1$$

$$i(\hat{\alpha_{hexa}}^n(2))=1$$ 
By symmetry $i(\hat{\alpha_{hexa}}^n(23))=i(\hat{\alpha_{hexa}}^n(32))=0,$ and finally
$$i(\hat{\alpha_{hexa}}^n(232))=i(\hat{\alpha_{hexa}}^n(22)=-1.$$

These words have for lengths $n,2n+1,2n+2, 2n+2,3n+3,3n+2$
We deduce that $b(n)$ can be computed by the euclidean division of $n$ by $6$. We obtain:
$$b(n)=\begin{cases}0, n=6k\\ 2, n=6k+1\\ 0, n=6k+2\\ 1, n=6k+3\\ 
1, n=6k+4\\ 1, n=6k+5\end{cases}$$
Thus the sum $\displaystyle\sum_{i=n}^{11+n}b(i)$ is a constant function of $n$.
\end{proof}

\begin{lemma}
We have:
$$\sum_{i=0}^{11}s(i)=76$$
\end{lemma}
\begin{proof}
It is a direct computation, we give here the different values of $p,s,b$.
\begin{tabular}{|c|c|c|c|c|c|c|c|c|c|c|c|c|c|c|c|c|}
\hline
n&0&1&2&3&4&5&6&7&8&9&10&11&12&13&14&15\\
\hline
p(n)&1&3&5&9&13&18&24&31&38&47&56&66&77&89&101&115\\
\hline
\end{tabular}

\medskip
\begin{tabular}{|c|c|c|c|c|c|c|c|c|c|c|c|c|c|c|c|}
\hline
n&0&1&2&3&4&5&6&7&8&9&10&11&12&13&14\\
\hline
s(n)&2&2&4&4&5&6&7&7&9&9&10&11&12&12&14\\
\hline
\end{tabular}
\medskip

\begin{tabular}{|c|c|c|c|c|c|c|c|c|c|c|c|c|c|c|c|c|}
\hline
n&0&1&2&3&4&5&6&7&8&9&10&11&12&13&14\\
\hline
b(n)&0&2&0&1&1&1&0&2&0&1&1&1&0&2&\\
\hline
\end{tabular}
\end{proof}

Now we compute the formula for the general case:
Consider the euclidean division of $n$ by $12$, $n=12q+r$. 
By Proposition \ref{calculhexa} we have.
$$s(n+12)-s(n)=10.$$
We deduce by an easy induction on $q$:
$$s(12q+r)-s(r)=10q.$$
$$s(12q+r)=10q+s(r).$$
Then this gives 
$$p(12q+r+1)=p(12q+r)+10q+s(r).$$
We write this equality for 12 consecutive numbers and sum the equalities. We obtain 
$$p(12(q+1)+r)=120q+\sum_{i=0}^{11}s(i)+p(12q+r).$$
Now preceding Lemma shows that we have
$$p(12(q+1)+r)=120q+76+p(12q+r).$$
Then if we denote $u_q=p(12q+r)$ for a fixed $r$, we obtain
$$u_{q+1}=120q+76+u_q.$$
Thus we have
$$u_q=\sum_{i=0}^{q-1}(120i+76)+p(r).$$
Finally we have
$$u_q=60q(q-1)+76q+p(r)=60q^2+16q+p(r).$$
$$p(n)=5(n-r)^2/12+ 16(n-r)/12+p(r).$$

This is equivalent to:
$$p(n)=\lfloor \frac{5n^2+16n+15}{12}\rfloor.$$

\subsection{Triangle}
\subsubsection{Bispecial words}

\begin{proposition}
Denote by $f$ the bijection map between the languages of the hexagon and the triangle. Then define $\hat{f}$ as the map 
$$\hat{f}(w)=1f(w)21\quad\forall w\in\{2,3\}^*$$
Then the bispecial words of the the language $L'$ of the triangle are the words 
$$\hat{f}(v),\quad v\in\mathcal{BL}_{hexa},$$
and the words 
$$\varepsilon, 1,11, 121,1121, 1211, 11211, 21121, 21211.$$  
\end{proposition}
\begin{proof}
By Corollary \ref{hexa-triang} there is a bijection between the languages associated to the triangle and the hexagon, thus we can prove the result for the regular hexagon, thus we refer to Lemma \ref{bisp-hexa}.
\end{proof}

\begin{corollary}
The bispecial words of the language $L'$ of the triangle are the words 
$$1(21)^n, 1(21)^n1(21)^n, 1(21)^n1(21)^{n+1},1(21)^{n+1}1(21)^n,$$
$$1(21)^n1(21)^n1(21)^n,1(21)^n1(21)^{n+1}1(21)^n \quad n\in\mathbb{N}.$$
\end{corollary}

\subsubsection{Proof of Theorem \ref{calccomp} for the triangle}
\begin{lemma}\label{calc-bisp-trian}
The complexity function of the language $L'$ associated to a triangle fulfills:
$$\displaystyle\sum_{i=n}^{n+11}b(i)=5.$$
\end{lemma}
\begin{proof}
 There are five types of bispecial words
\begin{itemize}
\item For the word $1(21)^n$ we see that there are four extended words. Thus this word fulfills 
$i(v)=1$.

\item For the words (and their mirror images) $1(21)^n1(21)^{n+1}$ there are three extensions: $ 21(21)^n1(21)^{n+1}2, 21(21)^n1(21)^{n+1}1$ and $11(21)^n1(21)^{n+1}1$, thus $i(v)=0$ for this word. 

\item For the words $1(21)^n1(21)^n$ there are four extensions, thus we have $i=1$.

\item For the words $1(21)^n1(21)^n1(21)^n$ there are two extensions: 
$$11(21)^n1(21)^n1(21)^n1, 21(21)^n1(21)^n1(21)^n2,$$ thus we have $i=-1$.

\item  For the words $1(21)^n1(21)^{n+1}1(21)^n$ there are two extensions: 
$$11(21)^n1(21)^{n+1}1(21)^n1, 21(21)^n1(21)^{n+1}1(21)^n2,$$ thus we have $i=-1$. 
\end{itemize}

Finally we obtain five sort of words of length $2n+1, 4n+2, 4n+4, 6n+3, 6n+5$ which are bispecial words. Between $n$ and $n+1$ there are $6,3,3,3*2,2,2$ words of each sort, so we have 
$s(n+12)-s(n)=6+3+0*6-2-2$.
\end{proof}

\begin{lemma}
The complexity of the language $L'$ satisfies
$$\displaystyle\sum_0^{11}s(r)=37.$$
\end{lemma}
\begin{proof}
By an easy computation we deduce:

\begin{tabular}{|c|c|c|c|c|c|c|c|c|c|c|c|c|c|c|}
\hline
n&0&1&2&3&4&5&6&7&8&9&10&11&12&13\\
\hline
s(n)&1&1&1&2&3&3&3&3&4&5&5&5&6&6\\
\hline
\end{tabular}
\end{proof}

Now we can prove the formula of the theorem.
We use Lemma \ref{calc-bisp-trian}, the method is the same as in the case of the square:
The relation $s(n+12)=5+s(n)$ gives 
$s(n)=s(r)+5q$, where $n=12q+r$, and $0\leq r<12$. 
We deduce 
$$\displaystyle\sum_{N=0}^{11}s(n+N)=37+5\displaystyle\sum_{N=0}^{11}[(n+N)/12].$$
$$p(n+12)-p(n)=37+5\displaystyle\sum_{N=0}^{11}[(n+N)/12].$$
$$p(n+12)-p(n)=37+60q+5\displaystyle\sum_{N=0}^{11}[(N+r)/12].$$
$$p(n+12)-p(n)=37+60q+5r.$$
The change of $n$ into $n-12$ does not change $r$
$$p(n)=p(r)+37q+\displaystyle\sum_{i=0}^{q-1}(60i+5r).$$
$$p(n)=30q(q-1)+37q+5\displaystyle\sum_{i=0}^{q-1}r+p(r).$$
$$p(n)=30q(q-1)+(37+5r)q+p(r).$$
$$p(n)=30q^2+(7+5r)q+p(r).$$
$$p(n)=30(\frac{n-r}{12})^2+(7+5r)\frac{n-r}{12}+p(r).$$
$$p(n)=\frac{5n^2}{24}+\frac{7n}{12}-\frac{5r^2}{24}-\frac{7r}{12}+p(r).$$


\section{Piecewise isometry of Tabachnikov}\label{sectab}
In this section we recall some results proved by Tabachnikov in \cite{Ta.95}.

\subsection{Definition and results}
Consider Figure \ref{figtab}.
We define a piecewise isometry $(Z,G)$ on the union of two triangles 
$$Z= AFC\cup HFE.$$ The two triangles are isoceles, the angle in $A$ equals $2\pi/5, AF=1$ and $AC=\varphi$ where $\varphi=\frac{1+\sqrt{5}}{2}$. 
The map $G:Z\mapsto Z$ is defined as follows:
\begin{itemize}
\item a rotation of center $O_1$ and angle $-3\pi/5$ which sends $C$ to $E$, on $AFC$. 
\item a rotation of center $O_2$ and angle $-\pi/5$ which sends $H$ to $C$ on $HFE$.
\end{itemize}
\begin{figure}[h]
\begin{center}
\includegraphics[width= 5cm]{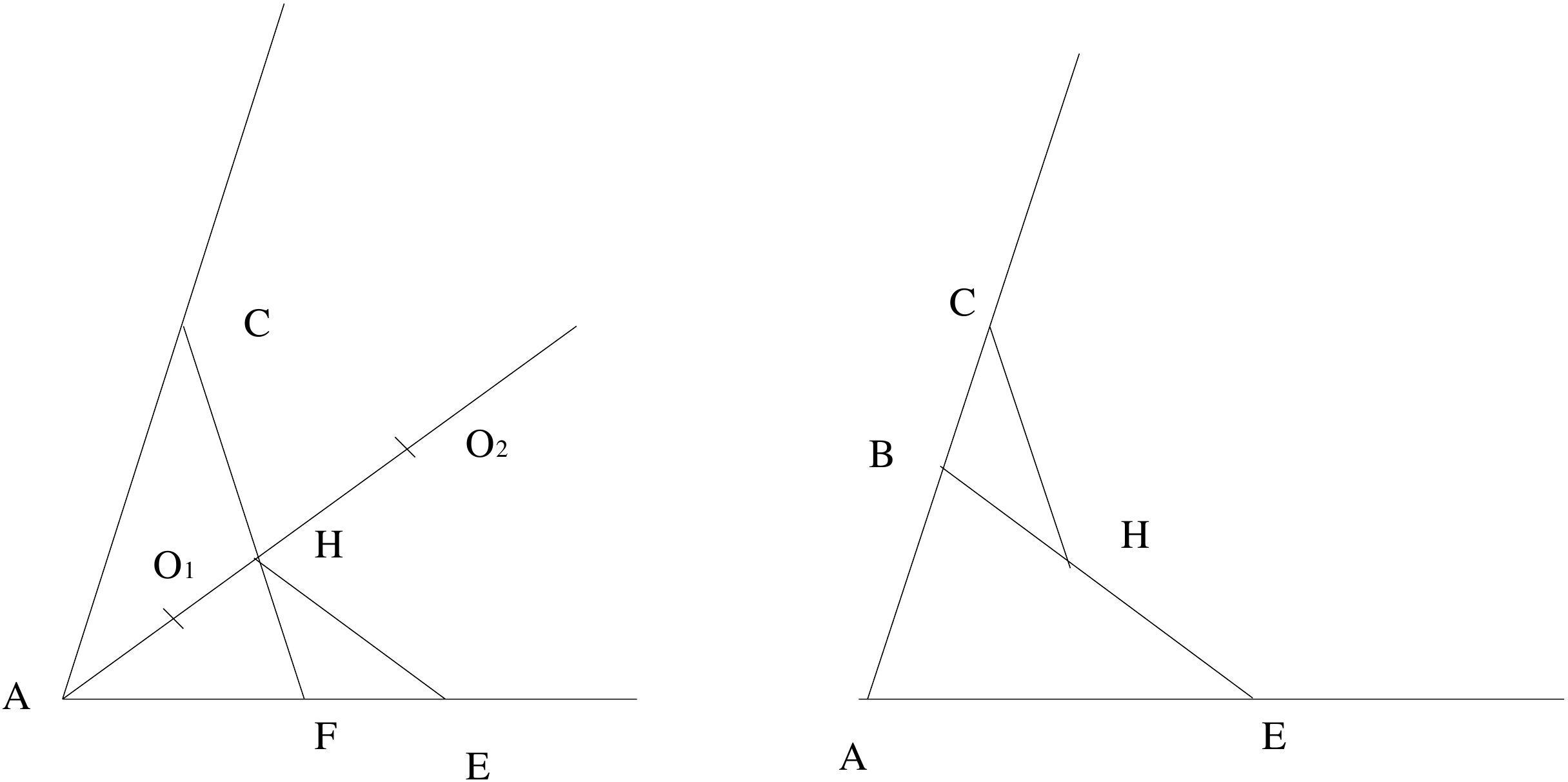}
\caption{Piecewise isometry $G$}\label{figtab}
\end{center}

\end{figure}

Denote by $a,b$ the two rotations, and $D$ the composition of the contraction, centered at $A$,  that takes $O_2$ to $O_1$  and the reflection over the line $(O_1O_2)$. Then we have on each triangle:
 $$\begin{cases}Da(x)=aababaaD(x)\quad \text{if}\quad x\in AFC,\\ Db(y)=aaaD(y)\quad  \text{if}\quad y\in HFE
 \end{cases}$$
 These equalities can be written as
 
 \begin{equation}
 DG(x)=G^7D(x), DG(y)=G^3D(y). 
 \end{equation}
 
 \begin{definition}\label{ptfixesigma}
Let $\sigma$ be the substitution:
 $$\sigma: \begin{cases}1\rightarrow 1121211 \\ 2\rightarrow 111\end{cases}$$
and let $u$ be its fixed point.
\end{definition}

\begin{definition}
 We denote by $V_{per}$ the set of periodic points for $G$ , and by $V_\infty$ the set $Z\setminus V_{per}$.
\end{definition}

 With Equation $1$ we deduce
 
 \begin{theorem}\cite{Ta.95}
 We have:
 \begin{enumerate}
\item If $x$ is a point with non periodic orbit under $G$, then the dynamical system $(O(x),G)$ is conjugated to $(O(u),S)$ where $S$ is the shift map, and $O(x)$ denotes the closure of the orbit of $x$.
\item A connected component of $V_{per}$ is a regular pentagon or a regular decagon.
\item Each point in a regular decagon has for coding an infinite word included in the shift orbit of $(\sigma^n(1))^\omega, n\in\mathbb{N}$. The points inside regular pentagons correspond to the words 
$(\sigma^n(12))^\omega, n\in\mathbb{N}$. 
\end{enumerate}
\end{theorem}

\begin{corollary}
The aperiodic points have codings included in the orbit $O(u)$.
\end{corollary}

\subsection{Link between $(Z,G)$ and the outer billiard outside the regular pentagon}
We will make more precize the statement of Lemma \ref{isom-newcoding}.
We use the same definitions, but the sector will be denoted by $V$.
\begin{definition}
The points refer to Figure \ref{pentaT1}. We define three sets
\begin{itemize}
\item $U_1$ is the triangle $AEB$.
\item $U_2$ is an infinite polygon with vertices $IBE$.
\item $U_3$ is a cone of vertex $I$.
\end{itemize}
\end{definition}

\begin{figure}
\begin{center}  
\includegraphics[width= 6cm]{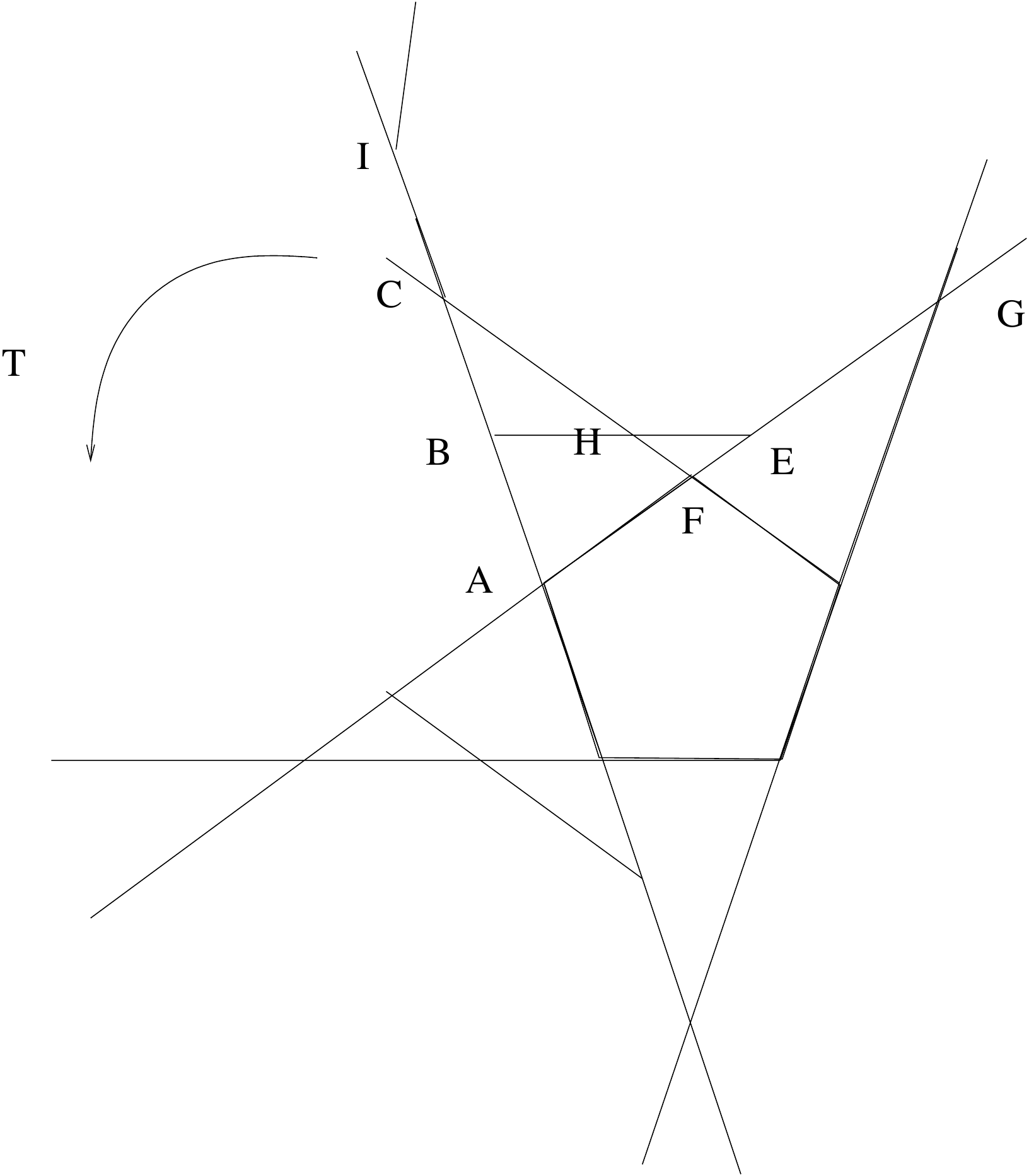}
\caption{Definition of $\hat{T}$} \label{pentaT1}
\end{center}
\end{figure}

\begin{lemma}\label{calcul-pentagone}
The map $\hat{T}$, see definition \ref{Tcone}, is defined on three subsets $U_1,U_2,U_3$.
$$V=U_1\cup U_2\cup U_3.$$
The images of $U_1,U_2,U_3$ by $\hat{T}$ verify the following properties:
 
\begin{itemize}
\item The first cell $U_1$ has for image the triangle $ACF$ and $\hat{T}_{|U_1}=RT$.
\item On the second cell $U_2$, we have $\hat{T}_{|U_2}=R^2T$, and the image of $U_2$ is an infinite polygon with vertices $CFG$. 
\item On the third set, we have $\hat{T}_{|U_3}=R^3T$, and the image of $U_3$ is a cone of vertex $G$.
\item The union of the triangles $ACF$ and $HFE$ is invariant by $\hat{T}$: $$U_1\cup (\hat{T}U_1\cap U_2).$$ 
\item The restriction of the map $\hat{T}$ to this invariant set is the map $(Z,G^{-1})$.
\item The set $[U_2\setminus (\hat{T}U_1\cap U_2)]\cup U_3$ is also invariant.
\end{itemize}
\end{lemma}

\begin{proof}
By definition we have 
$$TU_1\subset V_1, TU_2\subset V_2, TU_3\subset V_3.$$
We must verify $$\hat{T}(U_1\cup(\hat{T}U_1\cap U_2))\subset U_1\cup (\hat{T}U_1\cap U_2).$$ 
We have $\hat{T}U_1\subset U_1\cup U_2$ since $ACF\subset U_1\cup U_2$, thus $\hat{T}U_1\subset U_1\cup (\hat{T}U_1\cap U_2)$, and $\hat{T}(\hat{T}U_1\cap U_2)=\hat{T}(BHC)=HFE\subset U_1$.
\end{proof} 


\section{Dynamics on $U_2\cup U_3$}\label{secindupent}
The last point of the preceding lemma implies that we can restrict our study to a piece of $U_2\cup U_3$.
\subsection{First part}
\begin{lemma}\label{lemindupent}
The dynamics on the invariant set $[U_2\setminus (\hat{T}U_1\cap U_2)]\cup U_3$ is given by:
\begin{itemize}
\item The map restricted to $U_3$ is a rotation by angle $-\pi/5$.
\item The map restricted to $U_2\setminus Z$ is a rotation by angle $\pi/5$.
\end{itemize}
\end{lemma}

\begin{proof}
On the figure we see that for $\hat{T}$ restricted to $U_2$:

$$C\mapsto H, H\mapsto E, I\mapsto C, E\mapsto G, B\mapsto F$$
\begin{itemize}
\item The map on this infinite pentagon is a rotation centered in a point $O_2$ on the bissector of the sector. It is the {\bf same center} as the map $G$.
 \item The map on the infinite triangle is a rotation centered on a point $O_3$ of the bissector, but not inside the sector.
\end{itemize}
\end{proof}

\subsection{Notations}
We use the coding related to $\hat{T}$, see Subsection \ref{newcoding}.
Using Lemma \ref{calcul-pentagone} we see that on the invariant set $Z$ we have the same coding as in the piecewise isometry of Tabachnikov. 

\begin{definition}
We define the map 
$$\begin{array}{ccccc}
F&:&\{1;2;3\}^*&\to& Isom(\mathbb{R}^2)\\
&&v&\mapsto& F(v)
\end{array}$$
where $v=v_0\dots v_{n-1}$ is a finite word over the alphabet $\{1,2,3\}$, $F(v)$ is the composition of isometries:
$F(v)=F(v_{n-1})\circ\dots F(v_0)$. The map $F(i)$ coincides with  $\hat{T}$ on $U_i$ for $i=1\dots 3$.
\end{definition}
Remark that $F(2)$ and $F(3)$ are rotations of opposite angles. $F(2)$ has angle $\pi/5$ and $F(1)$ has angle equal to $3\pi/5$.

\subsection{Dynamics}
\subsubsection{Conjugaison}

The plane is identified with $\mathbb{C}$ with the origin in the vertex $A$ of $V_0$ and $F$ has for affix $1$, see Figure \ref{pentaT1}. If $F(v)$ is a rotation, then the affix of its center is denoted by $Z_v$.
Here we will find some relations between the map $F(i), i=1\dots 3$. 
After this, we will deduce some results on the coding words.
We recall the following classical fact
\begin{lemma}\label{lem-rot-centre}
Let $u_1,u_2$ two rotations of the plane of the same angle with centers points of affixes $z_1,z_2$. Then the translation vector $t$ wich conjuges the maps fulfills 
$$u_1\circ t=t\circ u_2, t=z_1-z_2.$$
\end{lemma}

\begin{lemma}\label{lem-rot-calcul}
With preceding notations we find
\begin{itemize}
\item $F(2)(z)=e^{i\pi/5}(z-(1+\varphi)e^{i\pi/5})+(1+\varphi)e^{i\pi/5}.$
\item $F(1)(z)=e^{3i\pi/5}z+1.$
\item $F(3)(z)=e^{-i\pi/5}z+\varphi+1-(\varphi+1) e^{i\pi/5}.$
\item $Z_1=\frac{1}{1+e^{-2i\pi/5}}, Z_2=(1+\varphi)e^{i\pi/5}, Z_3=-(1+\varphi)e^{i\pi/5}.$
\item $\cos{\pi/5}=\varphi/2, \cos{2\pi/5}=\frac{\varphi-1}{2}$.
\end{itemize}
\end{lemma}
\begin{proof}
We will refer to Figure \ref{pentaT1}. The map $F(2)$ is a rotation of angle $\pi/5$. 
Thus it fixes a decagon.
The decagon has for ray $R=\varphi$ and for side length $1$, then its center is 
$Z_2=(1+\varphi)e^{i\pi/5}$, thus we obtain the expression of $F(2)$. 

Now $F(3)$ has opposite angle, to obtain the constant, we remarked that $F(3)$ maps the point $I$ of coordinates $(\varphi+1) e^{i2\pi/5}$ to the point $G$ of coordinates $\varphi+1$.
$$F(3)(z)=e^{-i\pi/5}z+\varphi+1-(\varphi+1) e^{i\pi/5}.$$
Then its center is
$$Z_3=(1+\varphi)\frac{1-e^{i\pi/5}}{1-e^{-i\pi/5}}=-(1+\varphi)e^{i\pi/5}.$$

For $F(1)$ we see that it maps the point $A$ of coordinates $0$ to the point $F$ of coordinates $1$, thus:
$$F(1)(z)=e^{3i\pi/5}z+1.$$
Its center is 
$$Z_1=\frac{1}{1+e^{-2i\pi/5}}.$$
\end{proof}

\begin{lemma}\label{translation-calcul}
There exists a translation $t$ such that
$$F(223)\circ t=t\circ F(2).$$
$$F(2223223)\circ t=t\circ F(1).$$
$$F(3^{-1}2^{-1}3)\circ t=t\circ F(3).$$
\end{lemma}
\begin{proof}
The maps $F(223)$  and $F(2)$ are rotations of same angle thus they are conjugated by a translation. Denote this map by $t$. The same argument shows that $F(22223)$ and $F(1)$ are conjugated by a translation, it remains to prove it is the same vector. To finish the proof  we must do the same computation for the two rotations $F(2223223),F(1)$. We use Lemma \ref{lem-rot-calcul} in the following.
\begin{itemize}
\item $$F(2)^2(z)=e^{i2\pi/5}(z-(1+\varphi)e^{i\pi/5})+(1+\varphi)e^{i\pi/5}.$$
$$F(3)\circ F(22)(z)=e^{i\pi/5}(z-(1+\varphi)e^{i\pi/5})+(1+\varphi)+\varphi+1-(\varphi+1) e^{i\pi/5}.$$
$$F(3)\circ F(22)(z)=e^{i\pi/5}(z-(1+\varphi)e^{i\pi/5})+(1+\varphi)(2-e^{i\pi/5}).$$
$$F(3)\circ F(22)(z)=F(2)(z)+2(\varphi+1)(1-e^{i\pi/5}).$$
Thus its center is the point
$$Z_{223}=(1+\varphi)\frac{2-e^{i\pi/5}-e^{2i\pi/5}}{1-e^{i\pi/5}}$$
Lemma \ref{lem-rot-centre} gives 
$$t=Z_{223}-Z_2=2(1+\varphi).$$

\item We show here that $t=Z_{2223223}-Z_1.$
First remark that $F(2223223)=F(223)^2\circ F(2)$. Now we have
$$F(223)^2(z)=e^{2i\pi/5}(z-Z_{223})+Z_{223}.$$
Thus we obtain:
$$F(2223223)(z)=e^{2i\pi/5}(F(2)(z)-Z_{223})+Z_{223}.$$
Since we have 
$F(2)(z)=e^{i\pi/5}z-(1+\varphi)e^{i2\pi/5}+(1+\varphi)e^{i\pi/5}$ we deduce
$$F(2223223)(z)=e^{2i\pi/5}[e^{i\pi/5}z-(1+\varphi)e^{i2\pi/5}+(1+\varphi)e^{i\pi/5}-Z_{223}]+Z_{223}.$$
Now preceding item gives 
$$Z_{223}=(1+\varphi)(2+e^{i\pi/5}).$$
$$F(2223223)(z)=e^{2i\pi/5}[e^{i\pi/5}z-(1+\varphi)(e^{i2\pi/5}-e^{i\pi/5}+2+e^{i\pi/5})]+(1+\varphi)(2+e^{i\pi/5}).$$
$$F(2223223)(z)=e^{2i\pi/5}[e^{i\pi/5}z-(1+\varphi)(e^{i2\pi/5}+2)]+(1+\varphi)(2+e^{i\pi/5}).$$

$$F(2223223)(z)=e^{3i\pi/5}z+(1+\varphi)(2+e^{i\pi/5}-e^{i4\pi/5}-2e^{2i\pi/5}).$$




$$Z_{2223223}=(1+\varphi)\frac{2+e^{i\pi/5}-e^{i4\pi/5}-2e^{2i\pi/5} }{1-e^{i3\pi/5}}.$$
$$Z_{2223223}=(1+\varphi)\frac{2+e^{i\pi/5}+e^{-i\pi/5}-2e^{2i\pi/5} }{1+e^{-i2\pi/5}}.$$
$$Z_{2223223}=(1+\varphi)\frac{2+2\cos{\pi/5}-2e^{2i\pi/5}}{1+e^{-i2\pi/5}}.$$
$$Z_{2223223}=(1+\varphi)\frac{2+\varphi-2e^{2i\pi/5}}{1+e^{-i2\pi/5}}.$$
$$Z_{2223223}=\frac{2+\varphi+2\Phi+\varphi+1-2(1+\varphi)e^{2i\pi/5}}{1+e^{-i2\pi/5}}.$$
$$Z_{2223223}=\frac{3+4\varphi-2(1+\varphi)e^{2i\pi/5}}{1+e^{-i2\pi/5}}.$$
Now we compute the translation vector
$$Z_{2223223}-Z_1=\frac{3+4\varphi-2(1+\varphi)e^{2i\pi/5}-1}{1+e^{-i2\pi/5}}.$$
$$Z_{2223223}-Z_1=2\frac{1+2\varphi-(1+\varphi)e^{2i\pi/5}}{1+e^{-i2\pi/5}}.$$
$$Z_{2223223}-Z_1=2\frac{1+2\varphi-(1+\varphi)\cos{2\pi/5}-(1+\varphi)i\sin{2\pi/5}}{1+e^{-i2\pi/5}}.$$
$$Z_{2223223}-Z_1=2\frac{1+2\varphi-(1+\varphi)(\varphi-1)/2-(1+\varphi)i\sin{2\pi/5}}{1+e^{-i2\pi/5}}.$$
$$Z_{2223223}-Z_1=\frac{2+3\varphi-2i(1+\varphi)\sin{2\pi/5}}{1+e^{-i2\pi/5}}.$$
$$Z_{2223223}-Z_1=\frac{2+3\varphi-2i(1+\varphi)\sin{2\pi/5}}{1+\cos{2\pi/5}-i\sin{2\pi/5}}.$$
$$Z_{2223223}-Z_1=2\frac{2+3\varphi-2i(1+\varphi)\sin{2\pi/5}}{1+\varphi-2i\sin{2\pi/5}}.$$
$$Z_{2223223}-Z_1=2\frac{(1+\varphi)^2-2i(1+\varphi)\sin{2\pi/5}}{1+\varphi-2i\sin{2\pi/5}}.$$
$$Z_{2223223}-Z_1=2(1+\varphi).$$

The last equality is proved by the same method.
\end{itemize}
\end{proof}

\begin{lemma}
There exists a rotation $u$ such that
$$F(32222)\circ u=u\circ F(1).$$
$$F(322222)\circ u=u\circ F(12).$$
\end{lemma}

\begin{proof}
The rotations $F(32222)$ and $F(1)$ have the same angle, thus they are conjugated by a rotation. This rotation must exchange the two centers.
The same thing is true for $F(32^5),F(12)$, thus there exists only one rotation which conjugate both maps. Moreover we obtain 
$F(2)\circ u=u\circ F(2)$. This implies that $u$ has the same center as $F(2)$. Now the angle of $u$ is computed and we obtain $-3\pi/5$.


\end{proof}

\subsubsection{Words}
\begin{corollary}\label{crucial-translations}
\begin{itemize}
\item For any word $v$ of the language of $\hat{T}$ we have: $$F(\psi(v))\circ t=t\circ F(v), F(\xi(v))\circ u=u\circ F(v).$$
\item For the language of the map $\hat{T}$, we have equivalence between
\begin{itemize}
\item$v$ is the code of a periodic point.
\item $\psi(v)$ is the code of a periodic point.
\end{itemize}
\item For the language of the map $\hat{T}$, we have equivalence between
\begin{itemize}
\item$v$ is a periodic word.
\item $\xi(v)$ is a periodic word.
\end{itemize}
\end{itemize} 
\end{corollary}
\begin{proof}
\begin{itemize}
\item The equalities are fulfilled for $1,2,3$ by preceding Lemma. Thus the equalities are true since $1,2,3$ are generators of every word.
\item Assume $v$ is a periodic word. Let $m$ be a point in the cell of $v$. We have 
$F(v)m=m$. Now we have 
$F(\psi(v))\circ tm=t\circ F(v)m=tm$. We deduce that $tm$ is a periodic point for the word $\psi(v)$. since this equality is true for every point $m$, the word $\psi(v)$ is periodic. 
\end{itemize}
The other part of the proof is similar.
\end{proof}

\begin{corollary}
\begin{itemize}
\item The cells associated to the periodic words of period $\psi^{2k}(2), \psi^{2k+1}(1)$ are regular decagons. 
\item The cells associated to the periodic words of period $\psi^k(21)$ are regular pentagons.
\item The cells associated to the periodic words of period $\xi^{k}(1)$ are regular decagons.
\item The cells associated to the periodic words of period $\xi^{k}(12)$ are regular pentagons.
\end{itemize} 
\end{corollary}
\begin{proof}
\begin{itemize}
\item The rotation $F(2)$ is a rotation of angle $-\pi/5$. Thus it fixes a regular decagon with the same center as $F(2)$. This decagon corresponds to the cell associated to the word $2^\omega$. We have the same thing for the map $F(1)$.

\item In the same spirit the map $F(21)$ is a rotation of angle $-4\pi/5$. It fixes a pentagon. Then we apply preceding result of the same Lemma and we obtain the existence of the periodic words.

\item Since $F(1)$ is a rotation of angle $-3\pi/5$ we deduce, that there is an invariant decagon inside the set $U_1$. This decagon corresponds to the periodic word of period $1$. Now by preceding Proposition we deduce that the word $\xi(1)$ is also a period for a word of the language, it corresponds to the second ring of small decagons in the picture.
\end{itemize}
\end{proof}


\section{Proof of Theorem \ref{langages} for the regular pentagon}\label{seclangpent}
\subsection{Proof}
We recall the theorem before the proof
\begin{thm}
For the outer billiard outside the regular pentagon, the language $L'$ of the dynamics of $\hat{T}$ is the set of factors of the periodic words of the form $z^\omega$ for z$\in Z$, where
$$\bigcup_{n\in\mathbb{N}}\{\sigma^n(1), \sigma^n(12)\},$$
$$\bigcup_{n,m\in\mathbb{N}} \{\psi^m(2), \psi^m(2223), \psi^m\circ\sigma^n(1), \psi^m\circ\sigma^n(12)\},$$
$$\bigcup_{n,m\in\mathbb{N}} \{\psi^m\circ\xi\circ\sigma^n(1), \psi^m\circ\xi\circ\sigma^n(12)\}.$$

\end{thm}

\begin{remark}
In other terms we can split the language in different sets of periodic words with following periods:
\begin{itemize}
\item The words given by iteration of $\sigma$ on $1$ or $12$. 
\item The words obtained by the iterations of $2223$ or $2$ under $\psi$. 
\item The iterates of $\psi$ on the periodic words of the first class. They corresponds to coding of orbits of the following points: Take one point in $Z$, translate it by some power of $t$, see Lemma \ref{translation-calcul}.  

\item The words obtained by composition of $\xi$ and a power of $\psi$ on the first words. 

\end{itemize}
\end{remark}

We will use Corollary \ref{crucial-translations}. We will construct three invariant regions which will glue together and form a fundamental domain for the action of the translation $t$. The three sets are given by
\begin{itemize} 
\item The properties of the substitution $\psi$ imply the existence of a translation $t$ which is parallel to a side of the cone. Let $j$ be an integer, then the set $Z+jt$ is not invariant by $\hat{T}$. Now consider its orbit: $$\bigcup_{i\in \mathbb{N}}\hat{T}^i(Z+jt).$$
First by Lemma \ref{lemindupent} we know that its orbit is inside $U_2\cup U_3$. $Z$ is made of one big triangle and one small triangle. We claim that the five first iterations form an invariant set made of eight big triangles and three small ones. 
The symbolic dynamics inside this set is given by the composition of $\psi^j$ and 
$\sigma$. It is the first invariant ring.

\item Now we look at the second substitution $\xi$. Corollary \ref{crucial-translations} shows that there exists an invariant ring corresponding to the cells of the orbit under the shift of $(32222)^\omega$. The symbolic dynamics inside this set is given by the composition of $\xi$ and one iterate of $\sigma$. 

\item Moreover by Lemma \ref{lemindupent} there is an invariant polygon inside $U_2$, which is a regular decagon.
\end{itemize}
These three invariant rings glue together by trigonometric arguments.

Now we have a compact set inside $V$ with a symbolic description. These words are thus obtained as the $\mathbb{Z}^2$ sequences:
$$\psi^j\circ\sigma^i(1), \psi^j\circ\xi\circ\sigma^i(1),\psi^j(2223).$$ 

Now consider the parallelogram in $V$ with sides parallel to the axis of $V$ and of lengths $|t|$ and $\Phi$. Then consider the images of it by all the power of $t$. It forms a 
strip. 
Let $m$ be a point in $V$, then either it is in the strip or there exists an integer $n_0$ such that $\hat{T}^{n_0}m$ is inside this strip. Indeed outside the strip $\hat{T}$ acts as a rotation. Thus the dynamics of every point can be described with the preceding description.

\subsection{Examples}
Here we give some different examples of words coding some periodic orbits.

\begin{example}
We see on Figure \ref{sect-final} from the left to the right
\begin{itemize}
\item The first small decagon has coding $1^\omega$.
\item The second small decagon $(22223)^\omega=\xi(1)^\omega$.
\item The third small decagon is $(2223223)^\omega=\psi(1)^\omega$.
\item The first big decagon is $2^\omega$.
\item The second big decagon is $(223)^\omega=\psi(2)^\omega$.
\item The last: $22323^\omega=\psi(\xi(1))^\omega$ 
\item The first small pentagon has coding $(12)^\omega$.
\item The second small pentagon has coding $(222223)^\omega=\xi(12)^\omega$.
\item The small before the second big decagon $2223223223$.
\item The first big pentagon has coding $2223$.
\item The second big pentagon is $22322323$.
\end{itemize}
\end{example}

\begin{figure}
\begin{center}
\includegraphics[width=6cm]{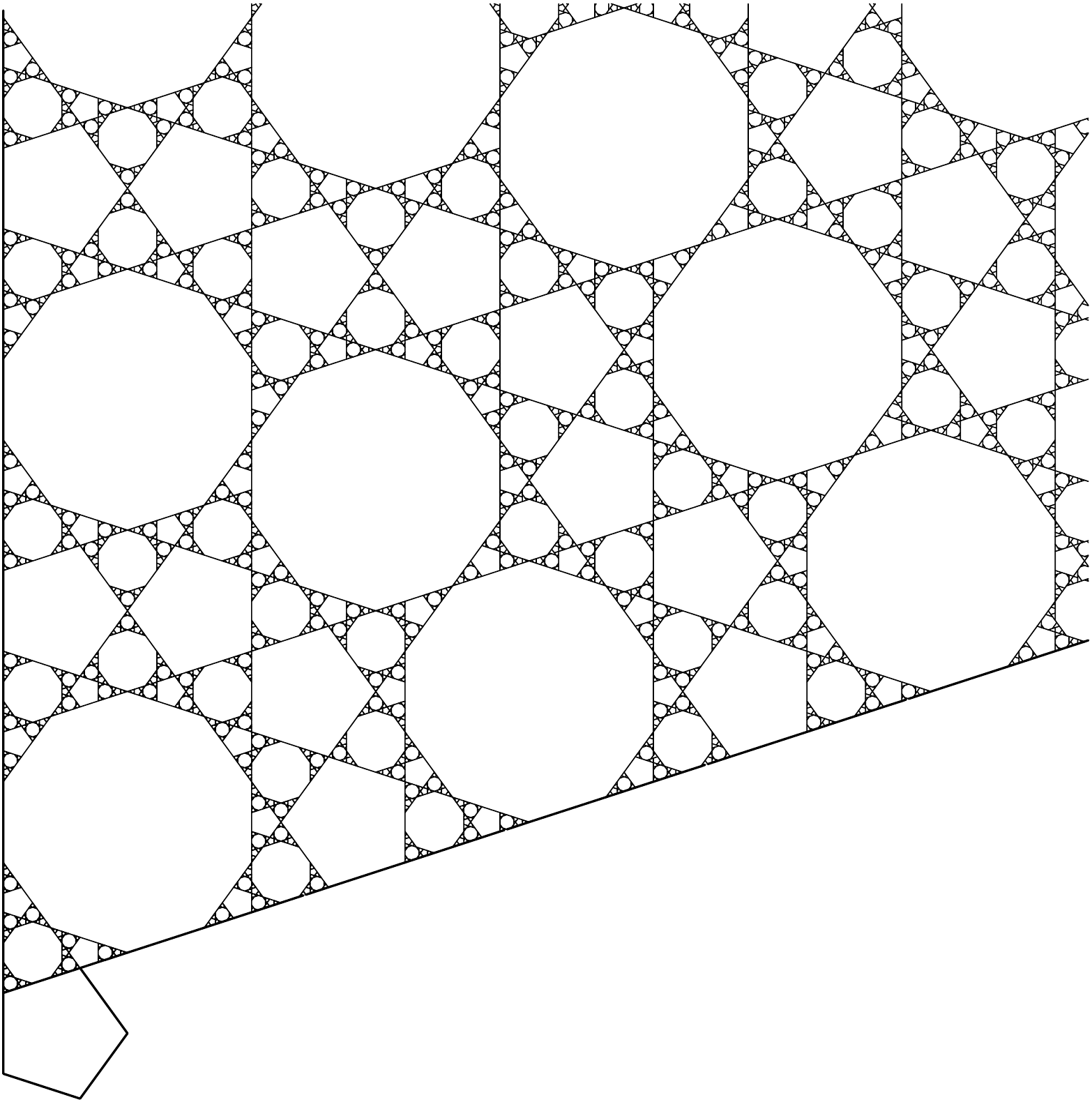}
\end{center}
\caption{Sector}\label{sect-final}
\end{figure}

\section{Bispecial words}\label{seccalccomppent}
In this Section we will describe the bispecial words of the language of the outer billiard map outside the regular pentagon. 

\begin{definition}\label{deffinalmorph}
We introduce different maps and words to simplify the statement of the result.

\begin{multicols}{3}

\begin{itemize}
\item $\Phi\begin{cases}1\mapsto 1\\ 2\mapsto 2\\ 3\mapsto 23 \end{cases}$
\item $\tilde{\psi}\begin{cases}1\mapsto 23232\\ 2\mapsto 32\\ 3\mapsto 3 \end{cases}$\item $\tilde{\xi}\begin{cases}1\mapsto 3222\\ 2\mapsto 2 \end{cases}$
\item $\tilde{\psi}_{|\{1,2\}^*}=\tilde{\beta}.$
\item $\tilde{\psi}_{|\{2,3\}^*}=\tilde{\chi}.$
\\
\item $\hat{\tilde{\chi}}(w)=\tilde{\chi}(w)3$ for all word $w$.
\item $\hat{\tilde{\xi}}(w)=222\tilde{\xi}(w)$ for all word $w$.
\item $\hat{\sigma}(w)=11\sigma(w)11$ for all word $w$.
\item $\hat{\Phi}(w)=\Phi(w)2$ for all word $w\in \{2,3\}^*$.
\\
\item $\hat{\tilde{\beta}}(w)=23232\tilde{\beta}(w)$ for all word $w$. 
\item $x_n=\hat{\sigma}^n(1)$ for all integer $n$.
\item $y_n=\hat{\sigma}^n(1111)$ for all integer $n$.
\item $z_n=\hat{\sigma}^n(12121)$ for all integer $n$.
\item $t_n=\hat{\sigma}^n(1^7)$ for all integer $n$.
\end{itemize}
\end{multicols}

\end{definition}
Remark that we have $\psi=\Phi\circ\tilde{\psi}\circ\Phi^{-1}, \xi=\Phi\circ\tilde{\xi}\circ\Phi^{-1}.$

The aim of this part is to prove 

\begin{proposition}\label{classfinbisp}
The bispecial words of the language $L'$ of the outer billiard outside the regular pentagon form $24$ families, according to preceding definition.
\begin{itemize}
\item The empty word $\varepsilon$, with $i(\varepsilon)=2$.

\item The word $2$, with $i(2)=0$.

\item The weak bispecial words are with $k,n \in\mathbb{N}$: 
$$
\hat{\Phi}(\hat{\tilde{\chi}}^k(2222)),
\hat{\Phi}(\hat{\tilde{\chi}}^k(22322)), 
\hat{\Phi}(\hat{\tilde{\chi}}^k(232232)),
\hat{\Phi}(\hat{\tilde{\chi}}^k(2323232)),$$
$$
\hat{\Phi}(\hat{\tilde{\chi}}^k\circ\hat{\tilde{\xi}}(z_n)),
\hat{\Phi}(\hat{\tilde{\chi}}^k\circ\hat{\tilde{\xi}}(t_n)),
\hat{\Phi}(\hat{\tilde{\chi}}^k\circ \hat{\tilde{\beta}}(z_n)), 
\hat{\Phi}(\hat{\tilde{\chi}}^k\circ \hat{\tilde{\beta}}((t_n)),$$
$$z_n,t_n$$

\item The strong bispecial words with $k,n \in\mathbb{N}$: 
$$
\hat{\Phi}(\hat{\tilde{\chi}}^k(2)),
\hat{\Phi}(\hat{\tilde{\chi}}^k(22)),
\hat{\Phi}(\hat{\tilde{\chi}}^k(222)),
\hat{\Phi}(\hat{\tilde{\chi}}^k(232)),
\hat{\Phi}(\hat{\tilde{\chi}}^k(23232)),
\hat{\Phi}(\hat{\tilde{\chi}}^k(3)),$$
$$
\hat{\Phi}(\hat{\tilde{\chi}}^k\circ\hat{\tilde{\xi}}(x_n)),
\hat{\Phi}(\hat{\tilde{\chi}}^k\circ\hat{\tilde{\xi}}(y_n)),
\hat{\Phi}(\hat{\tilde{\chi}}^k\circ\hat{\tilde{\beta}}(x_n)),
\hat{\Phi}(\hat{\tilde{\chi}}^k\circ\hat{\tilde{\beta}}(y_n)),$$

$$x_n,y_n
$$

\end{itemize}
\end{proposition}

\subsection{Notations}
We use Figure \ref{diagram} and define five languages $L_0,L_1, L_2, L_4,L_3$.
\begin{definition}
We denote $L_0,L_1, L_2, L_4,L_3$ the languages made respectiveley by factors of the following words:
\begin{itemize}
\item $L_0=\displaystyle\bigcup_{n\geq 0}\sigma^n(1)^\omega\cup\sigma^n(12)^\omega.$
\item $L_1=\tilde{\xi}(L_0).$
\item $L_2=\tilde{\beta}(L_0).$ 
\item $L_3=\displaystyle\bigcup_{m\geq 1} \tilde{\chi}^m(L_3)\cup \tilde{\chi}^m(L_2)\cup\tilde{\chi}^m(L_1).$
\item $L_4=2^\omega\cup (223)^\omega.$
\end{itemize}
\end{definition}

\begin{figure}
$$ \xymatrix{ (12)^\omega\ar[d]&1^\omega \ar[ld]\\
 L_0\ar@(ul,dl)[]_\sigma\ar[d]_{\tilde{\xi}}\ar[rd]_{\tilde{\beta}}\\ 
 L_1\ar[d]_{\tilde{\chi}}&L_2\ar[ld]_{\tilde{\chi}}\\
 L_4\ar@(ur,dr)[]_{\tilde{\chi}}\\
L_3\ar[u]_{\tilde{\chi}}
}$$
\caption{Construction of $L_\Phi$}\label{diagram}
\end{figure}
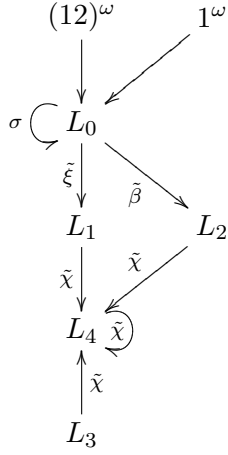

\subsection{Step one}
\subsubsection{Simplification of the problem}
First we use conjugaison by $223$ on $\psi$. It does not change the images of $1,2$, and the image of $3$ is $2^{-1}$. Thus we can assume that $\psi$ is defined with $\psi(3)=2^{-1}$.
Now we use conjugation by $2$ to change the maps $\psi,\xi$.
$$\psi_2=2^{-1}.\psi .2, \xi_2=2^{-1}.\xi. 2$$ 
$$\psi_2\begin{cases}1\mapsto 2232232\\ 2\mapsto 232\\ 3\mapsto 2^{-1} \end{cases}$$
$$\xi_2\begin{cases}1\mapsto 23222\\ 2\mapsto 2\\ \end{cases}$$

\begin{lemma}
The conjugaison does not change the language. In other terms the language is equal to the set of factors of the following words:
$$\bigcup_{n\in\mathbb{N}}\sigma^n(1)^\omega\cup\sigma^n(12)^\omega,$$
$$\bigcup_{n,m\in\mathbb{N}} \psi_2^n(2)^\omega\cup\psi_2^n(2223)^\omega\cup [\psi_2^m\circ\sigma^n(1)]^\omega\cup[\psi_2^m\circ\sigma^n(12)]^\omega,$$
$$\bigcup_{n,m\in\mathbb{N}} [\psi_2^m\circ\xi_2\circ\sigma^n(1)]^\omega\cup [\psi_2^m\circ\xi_2\circ\sigma^n(12)]^\omega.$$
\end{lemma}
\begin{proof} For any word $w$ we have the relation 
$\psi_2(w)=2^{-1}\psi(w)2$. 
Thus we have the relation $\psi^n(w)=2\psi_2^{n}(w)2^{-1}.$ It means that we have shifted the sequence, thus the languages are the same.
$$\psi_2(u)=S\psi(u).$$
\end{proof}

Thus we will keep the same notations $\psi,\xi$ for the new maps.
\begin{lemma}
The language $L'$ of the outer billiard map outside a regular pentagon is the union:
$$L'=Fact(\Phi(L_0\cup L_1\cup L_2\cup L_4\cup L_3)).$$
\end{lemma}
\begin{proof}
By using the conjugation by $\Phi$ and  Theorem \ref{langages} we obtain a description of our language. The map $\Phi^{-1}$ is given by 
$\begin{cases}1\mapsto 1\\ 2\mapsto 2\\ 3\mapsto 2^{-1}3 \end{cases}$. Thus we deduce that $\Phi^{-1}(L_0)=L_0$, and deduce the result.
\end{proof}

\begin{definition}
By preceding Lemma $L'$ is the union of $L_0$ and the image by $\Phi$ of a set. We denote  $L_\Phi=L_1\cup L_2\cup L_4\cup L_3$ so that 
$L'=\Phi(L_\Phi)\cup L_0.$
\end{definition}
The preceding operations can be summarized in Figure \ref{diagram}.

\subsubsection{Study of the map $\Phi$}
In this part we explain how to manage the map $\Phi$ and restrict the study to the language $L_\Phi$.
In order to prove this result we use a synchronization lemma
\begin{lemma}\label{lem-step-one}
If $w$ is a factor of the language $\Phi(L_\Phi)$, then there exists a unique triple $(s,v,p)$ such that $w=s\Phi(v)p$ with $s\in\{\varepsilon,3\}, v\in L_\Phi, p\in\{\varepsilon,2\}$, with the additional condition $v\in\mathcal{A}^*2\Rightarrow p=2.$
\end{lemma}
\begin{proof}
First we are only interested in the case where $w=s\Phi(v)$.
If $w$ begins by $1$, then it is clear that $w$ can be written in the form $\Phi(1v)$. If $w$ begins by $2$ then three possibilities appear for the beginning: 
$21$ thus we write $w=\Phi(21v)$; or $22$ and  $w$ can be written as $\Phi(2v)$; and last possibility $23$, then $w=\Phi(3v)$.
If $w$ begins by $3$, we do the same thing and remark that the only problem is if $w$ begins with $32$. In this case we have $s=3$. The first part of the lemma is proven. Now we consider words of the form $w=\Phi(v)p$.The proof is similar:  the uniqueness is a consequence of the proof.
\end{proof}

\begin{corollary}
A word $w$ is bispecial in the language $L'$ if and only if either 
\begin{itemize}
\item $1$ occurs in $w$ and $w$ is bispecial in $L_0$, with same index.
\item $3$ or $22$ occurs un $w$, and $w=\hat{\Phi}(v)$ with $v$ bispecial in $L_\Phi$
\item $w=\varepsilon$,
\item $w=2$.

\end{itemize}
\end{corollary}
\begin{proof}
We apply preceding Lemma to $w$. Now if $w$ is a bispecial word of $\Phi(L_\Phi)$, then $s=\varepsilon$, indeed the properties of $\Phi$ imply $13,33$ does not belong to $\Phi(L_\Phi)$.
Since the image of each letter by $\Phi$ ends with the same letter we remark that  $xw\in \Phi(L_\Phi)\Rightarrow xv\in L_\Phi$. There are two cases
\begin{itemize}
\item $p=\varepsilon$: then the extensions of $w$ must belong to $\Phi(L_\Phi)$. It implies $v1\in L_\Phi$ and either $v2$ or $v3$ belong to $L_\Phi$ (or both). 
\item If $p=2$, then $w1$ or $w2$ or $w3$ belong to $\Phi(L_\Phi)$. It implies $v21\in L_\Phi$ resp. $v22\in L_\Phi$ or $v23\in L_\Phi$ resp. $v3\in L_\Phi$.
\end{itemize} 
Then we can summarize this study in the four cases
\begin{itemize}
\item $w1,w2\in \Phi(L_\Phi)$, then $v21$ and $v22\quad\text{or}\quad v23\in L_\Phi$. So $v2$ is bispecial in $L_\Phi$.
\item $w1,w3\in \Phi(L_\Phi)$, then $v2, v3\in L_\Phi$, so $v$ is bispecial in $L_\Phi$.
\item $w2,w3\in \Phi(L_\Phi)$, then $v2, v3\in L_\Phi$, so $v$ is bispecial in $L_\Phi$.
\item $w1,w2,w3\in \Phi(L_\Phi)$, then both $v$ and $v2$ are bispecial in $L_\Phi$.
\end{itemize}
This finishes the proof.
\end{proof}

\begin{remark}
This lemma allows us to forget the map $\Phi$ until Section \ref{secfin} and to study the bispecial words of the language $L_\Phi$.
\end{remark}

\subsection{Abstract of the method}
The method in order to list the bispecial words is the following. We begin by the bispecial words which are not in the intersection of two of the languages $L_1, L_2, L_4, L_3$. For the language $L_4$ we prove that these words are images of bispecial words of $L_2\cup L_1$. Then we prove that bispecial words in $L_1\cup L_2$ are images of bispecial words in $L_0$, finally we list the bispecial words of $L_0$. Then it remains to treat the words which are in the intersection of two languages. 

\subsection{Different languages}
We will need the following result.
\begin{lemma}\label{lemtriv}
We have 
$$L_0\subset\{1,12\}^*, L_1\subset\{3222,32222\}^*, L_2\subset\{23,223\}^*,$$ 
$$ L_4\subset\{3,32\}^*, L_3\subset\{2,223\}^*$$
\end{lemma}
\begin{proof}
The proof just consists in the remark that $22$ does not appear in $L_0$. Thus in $L_2$, the word $32$ appears if $u$ contains $1$ and $322$ appears when $u$ contains $21$.
For $L_1$, the word $3222$ appears in $\xi(1)$, and the word $32222$ appears from the word $12$. From the image by $\tilde{\chi}$ of $32,3222,32222$ we deduce the result.
\end{proof}
\subsubsection{Language $L_4$}
\begin{proposition}
If $w$ is a non empty bispecial word of $L_4$, then we have $w=s\tilde{\chi}(v)p$ with $s=\varepsilon, p=3$, where $v\in L_\Phi$ is a bispecial word of $L_\Phi$ with the same extansions as $w$. 
\end{proposition}
\begin{proof}
By the definition of $L_4$ we have $w\in F(\chi(L_1\cup\dots L_3))$. By Lemma \ref{lemtriv} a bispecial word must begin and end with the letter $3$. Now since $L_4$ is built with the words $3$ and $32$ we can remark that these words are images of $3$ and $2$ by $\tilde{\chi}$. The last letter of $w$ can not be the image by $\tilde{\chi}$ of $3$ since it must be prolonged by $2$. 
\end{proof}

\begin{corollary}\label{cor-L3}
If $w$ is a bispecial word of $L_4$ there exists an integer $k$ such that 
$w=\hat{\tilde{\chi}}^k(v),$ with $v\in L_1\cup L_2\cup L_3$. Moreover $v$ is a bispecial word in $L_\Phi$ with the same multiplicity. 
\end{corollary}
\begin{proof}
By preceding result $w=\hat{\tilde{\chi}}(v)$ and $v$ is a bispecial word. Thus if $v$ is not empty, we deduce $w=\hat{\tilde{\psi}}^2(v').$ We do the same thing until $v'$ belongs 
to $L_1\cup L_2\cup L_3$. This must happen since the lengths decrease.
\end{proof}

\subsubsection{Language $L_3$}
\begin{lemma}\label{lem-L4}
The bispecial words in this language are:
$$\{\varepsilon, 2, 22\}.$$
\end{lemma}
\begin{proof}
Proof left to the reader.
\end{proof}

\subsubsection{Language $L_1$}
\begin{lemma}\label{lem-L1}
A bispecial word $w$ of $L_1$ fulfills one of the following facts.
\begin{itemize}
\item $w\in\{\varepsilon, 2, 22, 222, 2222,22322\}$.
\item $w=\hat{\tilde{\xi}}(v)$ where $v$ is bispecial in $L_0$, $v\neq \varepsilon$.  
\end{itemize}
\end{lemma}
\begin{proof}
First it is clear that a bispecial word in $L_1$ must begin and end by $2$. If $w$ is not a factor of $232232$, then assume $w$ does not contains $222$ as factor. By Lemma \ref{lemtriv}, $L_1\subset\{2,322\}^*$ we deduce that $w$ is a factor of $(223)^\omega$. Thus we have $w=(223)^k22$, and we deduce $k=1$ since $w\notin F(232232)$.
It remains one case corresponding to $222\in F(w)$. Then by definition of $L_1$ we have $w\in \tilde{\xi}(L_0)$. By Lemma \ref{lemtriv} the number of consecutive $2$ is equal to $3$ or $4$, moreover the letter $3$ is isolated. Then if $w$ is a bispecial word we deduce that $w$ begins with $2$. Now $2w$ must be a word of the language. We deduce that $w$ begins with three letters $2$. Now we remark that $3222$ is equal to $\hat{\xi}(1)$. This implies that the suffixe of $w$ which prolong $222$ is the image of one word by $\hat{\xi}$.
To finish the proof it remains to list the bispecial words factors of $232232$.
\end{proof}

\subsubsection{Language $L_2$}
\begin{lemma}\label{lem-L2}
A bispecial word $w$ of $L_2$ fulfills one of the following facts.
\begin{itemize}
\item $w \in F(23232).$
\item $w=\hat{\tilde{\beta}}(v)$ where $v$ is bispecial in $L_0$, $v\neq \varepsilon$.
\end{itemize}
\end{lemma}
\begin{proof}
First it is clear that a bispecial word in $L_2$ must begin and end by $2$.  By Lemma \ref{lemtriv}, we obtain $L_2\subset\{23232,2323232\}^*$. We split the proof in two cases: $23232\in F(w)$ or $23232\notin F(w)$. For the first case since $w$ is a bispecial word we have $2w\in L_2$. Then Lemma \ref{lemtriv} shows that $22$ can only be extended by $3232$, thus $w$ begins with $23232$, and $w=\hat{\tilde{\beta}}(v)$.
 \end{proof}

\subsubsection{Language $L_0$}
\begin{lemma}\label{propu}
The language $L_0$ fulfills the three following properties:
\begin{itemize}
\item $11211\notin L_0$. 
\item $22\notin L_0$.
\item Three consecutive occurences of $2$ are of the forms:
$21^l212$ or $2121^l2$ with $l\in\{1,4;7\}$.
\end{itemize}

The followings words  are bispecial words of $L_0$:
\begin{itemize}
\item $1^i$ for $i=0,2,3,5,6$ and it is an ordinary bispecial word: $i(1^i)=0$. Thus they do not modify the complexity.
\item $1^i$ for $i=1,4$ and it is a strong bispecial word: $i(1)=i(1^4)=1$.
\item $12121,1^7$ are weak bispecial words: $i(12121)=i(1^7)=-1$.
\end{itemize}

\end{lemma}
The proof is left to the reader.

\begin{lemma}\label{lem-L0}
We have different cases for a non-ordinary bispecial word $w$ of $L_0$.
\begin{itemize}
\item $w=1^n, n\in\{1,4,7\}$.
\item $w=12121$.
\item $w=\hat{\sigma}(v)$ where $v$ is a non-ordinary bispecial word of $L_0$ and $i(w)=i(v)$.
\end{itemize}
\end{lemma}
\begin{proof}
First we consider the words without $2$, then the word is a power of $1$, and preceding Lemma shows the different possibilities. Now if the word contains only one letter $2$, then the word has the form $w=1^m21^n$, now the fact that $11211\notin L_0$ (see preceding Lemma) shows that the only  possibility is $121$ which is ordinary. 
It remains the case where $w$ contains at least two letters $2$. Then either $w$ contains $212$ or the word is a factor of $\sigma(L_0)$. We have different subcases for a bispecial word $w=1^m2\dots 21^n$ of $L_0$ factor of $u$ due to preceding Lemma. 
\begin{itemize}
\item $m=1$, then the word $1w=112\dots 21^n$ must belong to the language. This implies that $w=1212\dots 21^n$, but the fact that $2w$ exists implies now that $w=12121$.
\item $m\in\{4,7\}$, then $m=7$ is cleary impossible. It remains one case which can be written by symmetry $w=11112\dots 21111$. An easy argument of synchronization finishes the proof.
\end{itemize}
\end{proof}

The preceding lemma implies that the bispecial words of $L_0$ are of the form
\begin{corollary}\label{cor-L0}
For the long bispecial words there are four famillies of words.
\begin{itemize}
\item $x_n=\hat{\sigma}^n(1), n\in\mathbb{N}$, $i(\hat{\sigma}^n(1))=1$. 
\item $y_n=\hat{\sigma}^n(1^4), n\in\mathbb{N}$, $i(\hat{\sigma}^n(1^4))=1$. 
\item $z_n=\hat{\sigma}^n(12121),  n\in\mathbb{N}$, $i(\hat{\sigma}^n(12121))=-1$.
\item $t_n=\hat{\sigma}^n(1^7),  n\in\mathbb{N}$, $i(\hat{\sigma}^n(1^7))=-1$.
\end{itemize}

The two first famillies are made of strong bispecial words, the two last are weak bispecial words.
\end{corollary}

\subsection{Intersection of languages}
We interest in the words which belong to different languages.

In the following Lemma we denote by $L_{ijk}$ the language intersection of the languages $L_i, L_j$ and $L_k$ for $i,j\in \{1\dots 4\}$.
\begin{lemma}
The words which belong to at least two languages are:
\begin{itemize}
\item $22\in L_{124},23\in L_{1234}, 32\in L_{1234}.$
\item $222\in L_{14}, 223\in L_{124}, 232\in L_{1234}, 322\in L_{124}, 323\in L_{23}$
\item $2222\in L_{14}, 2232\in L_{124}, 2322\in L_{124}, 2323\in L_{23}, 3223\in L_{24}, 3232\in L_{23}.$
\item $22322\in L_{14}, 23223\in L_{24}, 23232\in L_{23}, 32232\in L_{24}, 32323\in L_{23}.$
\item $232232\in L_{24}, 232323\in L_{23}, 323232\in L_{23}.$
\item $2323232\in L_{23}, 3232323\in L_{23}.$
\end{itemize}
\end{lemma}
\begin{proof}
The proof consists in the following method. We consider the words by family of different lengths. When we have listed all the words of a given length $i$, we consider the words of length $i+1$ which contain one of the preceding words as prefix or sufix. Then we use Lemma \ref{lemtriv} to verify if this word is in two languages.  This allows us to obtain the first list, after this it remains to look at the bispecial words.
\end{proof}

\begin{corollary}\label{lem-int}
We have
$$L_1\cap L_0=\{\varepsilon,2\}.$$
$$L_1\cap L_4=F(232).$$
$$L_1\cap L_3=F(22322)\cup \{222,2222\}$$
$$L_4\cap L_3=F(232).$$
$$L_2\cap L_3=F(232232).$$
$$L_2\cap L_4=F(2323232)$$
\end{corollary}
\begin{proof}
The proof is left to the reader.
\end{proof}

\begin{corollary}\label{cor-int-lang}
The bispecial words of $L_\Phi$ belonging to at least two of the languages $L_1, L_2, L_3, L_4$  are
\begin{itemize}
\item The ten strong bispecial words,
$$\varepsilon, 2, 3, 22, 33, 222, 232, 323, 23232, 32323.$$
\item The four weak bispecial words :
 $$ 2222, 22322, 232232, 2323232.$$
\end{itemize}
\end{corollary}

\subsection{Proof of Proposition \ref{classfinbisp}}
First Lemma \ref{lem-step-one} implies that a bispecial word $w\in L'$ can be written as $w=\hat{\Phi}(v)$ where $v\in L_\Phi$ is a bispecial word. Now we are interested in a bispecial word $v$ in $L_\Phi$. Several cases appear
\begin{itemize}
\item If $v\in L_0$, then Corollary \ref{cor-L0} shows that $v$ is inside four famillies of words. 
\item If $v\in L_1$ then Lemma \ref{lem-L1} implies that $v=\hat{\tilde{\xi}}(v')$ with $v'\in L_0$ or $v$ is element of a finite familly. Thus the preceding point completes the list of bispecial words of $L_1$.
\item If $v\in L_2$, then Lemma \ref{lem-L2} implies that except for a finite list of words, we can write $v=\hat{\tilde{\beta}}(v')$ with $v'\in L_0$.
\item If $v\in L_3$ then Lemma \ref{lem-L4} gives the complete list of bispecial words.
\item If $v\in L_4$, then by Corollary \ref{cor-L3} we know that $v=\hat{\tilde{\psi}}^k(v')$ where $v'\in L_1\cup L_2\cup L_3$. 
\item Corollary \ref{cor-int-lang} and the preceding points allow us to finish the proof.
\end{itemize}


\section{Proof of Theorem \ref{calccomp}}\label{secfin}
\subsection{Method}
We will use Proposition \ref{classfinbisp} to compute the different lengths of these bispecial words. 
We must compute the length of words obtained as iterations of one map on one word. We explain the method for $\hat{\tilde{\psi}}^n(v)$. We will use abelianization of the morphism $\psi$ and deduce the matrix of $\hat{\psi}$. Then we must compute the $n$ power of this matrix $\hat{A}$ and compute the norm of $\hat{A}^nw$ where $w$ is the column vector, abelianization of the word $w$.

\begin{definition}\label{defmatricessubst}
\begin{itemize}
\item The matrices of the following abelianization morphisms are denoted by 
$\begin{array}{|c|c|c|c|c|c|c|c|c|}
\hline
 \tilde{\chi}&\hat{\tilde{\chi}}&\tilde{\xi}&\hat{\tilde{\xi}}&\tilde{\beta}&\hat{\tilde{\beta}}&
 |\Phi|& |\hat{\Phi}|\\
\hline
A&\hat{A}&C&\hat{C}&B&\hat{B}&L&\hat{L}\\
\hline
\end{array}
$. 
\item The abelianization of the words $x_n, y_n, z_n, t_n$ are denoted by $X_n, Y_n, Z_n, T_n$.
\end{itemize}
\end{definition}

\begin{remark}
The matrices related to $\tilde{\chi},\hat{\tilde{\chi}},\tilde{\xi},\hat{\tilde{\xi}},\tilde{\beta},\hat{\tilde{\beta}},|\Phi|, |\hat{\Phi}|$ are the following:
 $$A=\begin{pmatrix}1&0\\1&1\end{pmatrix}, \hat{A}=\begin{pmatrix}1&0&0\\1&1&1\\ 0&0&1\end{pmatrix}, B=\begin{pmatrix}3&1\\2&1\end{pmatrix},$$ 
$$\hat{B}=\begin{pmatrix}3&1&3\\2&1&2\\0&0&1\end{pmatrix},   
C=\begin{pmatrix}3&1\\ 1&0\end{pmatrix},$$
$$\hat{C}=\begin{pmatrix}3&1&3\\ 1&0&0\\0&0&1\end{pmatrix},$$ 
$$L=\begin{pmatrix}1&2&0\end{pmatrix}, \hat{L}=\begin{pmatrix}1&2&1\end{pmatrix}.$$
The map $\hat{\psi}$ is always applied to a word defined on the alphabet $\{2,3\}$, thus we can consider $L$ as a $3$ dimensional vector.
\end{remark}

\subsection{Computation}
According to Proposition \ref{classfinbisp} we will compute 
the following vectors, for every integer $k$:
\begin{itemize}
\item For the weak bispecial words:

$$\hat{L}\hat{A}^k\begin{pmatrix}4\\0\\1\end{pmatrix}, \hat{L}\hat{A}^k\begin{pmatrix}4\\1\\1\end{pmatrix},
\hat{L}\hat{A}^k\begin{pmatrix}4\\2\\1\end{pmatrix},
\hat{L}\hat{A}^k\begin{pmatrix}4\\3\\1\end{pmatrix},$$
$$
\hat{L}\hat{A}^k\hat{C}(Z_n),
\hat{L}\hat{A}^k\hat{C}(T_n),
\hat{L}\hat{A}^k(BZ_n),
\hat{L}\hat{A}^k(BT_n),$$
$$Z_n,T_n.$$

\item For the strong bispecial words:

$$\hat{L}\hat{A}^k\begin{pmatrix}1\\0\\1\end{pmatrix},
\hat{L}\hat{A}^k\begin{pmatrix}2\\0\\1\end{pmatrix},
\hat{L}\hat{A}^k\begin{pmatrix}3\\0\\1\end{pmatrix},
\hat{L}\hat{A}^k\begin{pmatrix}2\\1\\1\end{pmatrix},
\hat{L}\hat{A}^k\begin{pmatrix}3\\2\\1\end{pmatrix},
\hat{L}\hat{A}^k\begin{pmatrix}0\\1\\1\end{pmatrix},$$
$$
\hat{L}\hat{A}^k\hat{C}(X_n),
\hat{L}\hat{A}^k\hat{C}(Y_n),  
\hat{L}\hat{A}^k(BX_n),
\hat{L}\hat{A}^k(BY_n),$$
$$X_n,Y_n.$$
\end{itemize}

\subsection{Length of bispecial words of $L_0$}
\begin{lemma} 
The vectors $X_n, Y_n, Z_n, T_n$, see Definitions \ref{deffinalmorph} and \ref{defmatricessubst}, are of the form
\begin{itemize}
\item $X_n=\frac{1}{35}\begin{pmatrix}54.6^n-5(-1)^n-14\\ 18.6^n+10.(-1)^n-28\\ 35\end{pmatrix}.$
\item $Y_n=\frac{1}{35}\begin{pmatrix}144.6^n+10(-1)^n-14\\ 48.6^n-20.(-1)^n-28\\ 35\end{pmatrix}.$
\item $Z_n=\frac{1}{35}\begin{pmatrix}144.6^n-25(-1)^n-14\\ 48.6^n+50.(-1)^n-28\\ 35\end{pmatrix}.$
\item $T_n=\frac{1}{35}\begin{pmatrix}234.6^n+25(-1)^n-14\\ 78.6^n-50.(-1)^n-28\\ 35\end{pmatrix}.$
\end{itemize}
\end{lemma}
\begin{proof}
With Lemma \ref{lem-L0} and Corollary \ref{cor-L0} we deduce that the abelianizations of the bispecial words of $L_0$ are of the form
$l_i=M^nv_i, i=1\dots 4$ where 
$M=\begin{pmatrix}5&3&4\\ 2&0&0\\ 0&0&1\end{pmatrix}$ and $v_i$ are the vectors:
$\begin{pmatrix}1\\0\\1\end{pmatrix}, \begin{pmatrix}4\\0\\1\end{pmatrix}, \begin{pmatrix}3\\2\\1\end{pmatrix}, \begin{pmatrix}7\\0\\1\end{pmatrix}.$
Now the matrix $M$ can be reduced to a diagonalized matrix with coefficients $-1,1,6$.
A simple calculus gives the expression of $M^n$. 
\end{proof}

\subsection{Iterations and lengths}
\begin{lemma}\label{lem-calcul-iter}
For all integer $k\geq 1$ we have

$$\hat{L}\hat{A}^k\begin{pmatrix}a\\b\\1\end{pmatrix}=k(2a+2)+2b+a+1,$$
$$\hat{L}\hat{A}^k\hat{B}\begin{pmatrix}a\\b\\1\end{pmatrix}=k(6a+2b+8)+7a+3b+8,$$
$$\hat{L}\hat{A}^k\hat{C}\begin{pmatrix}a\\b\\1\end{pmatrix}=k(6a+2b+8)+5a+b+4.$$
\end{lemma}
\begin{proof}
We remark that the minimal polynomial of $\hat{A}$ is $(X-1)^2$. We deduce the formula $\hat{A}^k=k\hat{A}-(k-1)Id$. Then we have for $k\geq 1$,
 
$$\hat{A}^k=\begin{pmatrix}1&0&0\\k&1&k\\0&0&1\end{pmatrix}.$$ We deduce that the word $\hat{\tilde{\chi}}^k(v)$ has the following abelianization
$\begin{pmatrix}a\\(a+1)k+b \\ 1\end{pmatrix}$ where the word $v$ has coordinates $v=\begin{pmatrix}a\\ b\\1\end{pmatrix}$.

In a similar way we obtain the matrix of $\hat{\tilde{\chi}}^k\hat{\tilde{\beta}}$, first we compute the matrix $\hat{B}$ of $\beta$, then we deduce
$$\hat{A}^k\hat{B}=
\begin{pmatrix}3&1&3\\ 3k+2&1+k&4k+2\\ 0&0&1\end{pmatrix}.$$
We deduce that the word $\hat{\tilde{\chi}}^k\hat{\beta}(v)$ has the following abelianization
$$\begin{pmatrix}3a+b+3\\(3a+b+4)k+2a+b+2 \\ 1\end{pmatrix},$$ where the word $v$ has coordinates $v=\begin{pmatrix}a\\ b\\1\end{pmatrix}$.

Finally we obtain
$$\hat{A}^kC=\begin{pmatrix}3&1&3\\ 3k+1&k&4k\\0&0&1\end{pmatrix}.$$
We deduce that the word $\hat{\tilde{\chi}}^k\hat{\xi}(v)$ has the following abelianization
$\begin{pmatrix}3a+b+3\\(3a+b+4)k+a \\ 1\end{pmatrix}$ where the word $v$ has coordinates $v=\begin{pmatrix}a\\ b\\1\end{pmatrix}$.
\end{proof}

\subsection{Formula}
Now we can use the formula for the computation of the lenghts of the bispecial words of $L_0$. In the following we denote 
$X_i=PD^nP^{-1}v_i$.
Finally we obtain

\begin{proposition}
The bispecial words of the language are of the following form, 
with $n,k\in\mathbb{N}\cup\{0\}$:
\begin{itemize}
\item The length of the weak bispecial words is of the form
$$\begin{cases}
10k+5\\
10k+7\\
10k+9\\
10k+11\\
\frac{48.6^n(20k+24)+14(10k+7)-25(-1)^n(2k+1)}{35}\\
\frac{78.6^n(20k+16)+14(10k+3)+25(-1)^n(2k+1)}{35}\\
\frac{48.6^n(20k+16)+14(10k+3)-25(-1)^n(2k+3)}{35}\\
\frac{78.6^n(20k+24)+14(10k+7)+25(-1)^n(2k+3)}{35}\\
\frac{192.6^n+25(-1)^n-42}{35}\\
\frac{312.6^n-25(-1)^n-42}{35}
\end{cases}
$$
\item The lengths of the strong bispecial words is of the form
$$\begin{cases}
4k+2\\
6k+3\\
8k+4\\
6k+5\\
8k+8\\
2k+3\\
\frac{18.6^n(20k+24)+14(10k+7)-5(-1)^n(2k+1)}{35}\\
\frac{48.6^n(20k+24)+14(10k+7)+10(-1)^n(2k+1)}{35}\\
\frac{18.6^n(20k+16)+14(10k+3)-5(-1)^n(2k+3)}{35}\\
\frac{48.6^n(20k+16)+14(10k+3)+10(-1)^n(2k+3)}{35}\\
\frac{4.18.6^n-42+5(-1)^n}{35}\\
\frac{4.8.6^n-42+10(-1)^n}{35}
\end{cases}$$
\end{itemize}
\end{proposition}

The next subsections split the proof in two parts.

\subsubsection{Weak bispecial words}
We use Lemma \ref{lem-calcul-iter} and obtain the formulas.
\begin{itemize}
\item $\hat{L}\hat{A}^k\begin{pmatrix}4\\0\\1\end{pmatrix}=10k+5, 
\hat{L}\hat{A}^k\begin{pmatrix}4\\1\\1\end{pmatrix}= 10k+7.$
\item $\hat{L}\hat{A}^k\begin{pmatrix}4\\2\\1\end{pmatrix}=10k+9, 
\hat{L}\hat{A}^k\begin{pmatrix}4\\3\\1\end{pmatrix}= 10k+11.$

\item $\hat{L}\hat{A}^k\hat{B}Z_n=\frac{1}{35}[48.6^n(20k+24)-25(-1)^k(2k+1)+14(10k+7)].$

\item $\hat{L}\hat{A}^k\hat{B}T_n=\frac{1}{35}[6.13.6^n(20k+24)+25(-1)^n(2k+1)+14(10k+7)].$

\item $\hat{L}\hat{A}^k\hat{C}Z_n=\frac{1}{35}[48.6^n(20k+16)-25(-1)^n(2k+3)+14(10k+3)].$

\item $\hat{L}\hat{A}^k\hat{C}T_n=\frac{1}{35}[6.13.6^n(20k+16)+25(-1)^n(2k+3)+14(10k+3)].$
\end{itemize}
The two last formula correspond to $Z_n, T_n$.

\subsubsection{Strong bispecial words}
By the same method we obtain
\begin{itemize}
\item $\hat{L}\hat{A}^k\begin{pmatrix}1\\0\\1\end{pmatrix}=4k+2,
\hat{L}\hat{A}^k\begin{pmatrix}2\\0\\1\end{pmatrix}=6k+3.$
\item $\hat{L}\hat{A}^k\begin{pmatrix}3\\0\\1\end{pmatrix}=8k+4,
\hat{L}\hat{A}^k\begin{pmatrix}2\\1\\1\end{pmatrix}=6k+5.$
\item $\hat{L}\hat{A}^k\begin{pmatrix}3\\2\\1\end{pmatrix}=8k+8,
\hat{L}\hat{A}^k\begin{pmatrix}0\\1\\1\end{pmatrix}=2k+3.$

\item $\hat{L}\hat{A}^k\hat{B}X_n\frac{1}{35}[18.6^n(20k+24)-5(-1)^n(2k+1)+14(10k+7)].$

\item $\hat{L}\hat{A}^k\hat{B}Y_n=\frac{1}{35}[48.6^n(20k+24)+10(-1)^n(2k+1)+14(10k+7)].$

\item $\hat{L}\hat{A}^kCX_n=\frac{1}{35}[18.6^n(20k+16)-5(-1)^n(2k+3)+14(10k+3)].$

\item $\hat{L}\hat{A}^kCY_n=\frac{1}{35}[48.6^n(20k+16)+10(-1)^n(3+2k)+14(10k+3)].$
\end{itemize}
The two last formula correspond to $X_n, Y_n$.

\begin{remark}
In fact the formulas of the length could be reduced by assuming that $n$ is an integer greatest than $-1$. In this case the formulas $8k+4,8k+8, 6k+3, 6k+5$ can be erased, since they appear as particular cases of other formulas. But there is no algebraic sense to $x_{-1}$, thus we do not use it.
\end{remark}
\subsection{Number of bispecial words}
\begin{lemma}
Denote $g(n,k)=6^n(ak+b)+ck+d+(-1)^n(ek+f)$ for some fixed rationals numbers $a,b,c,d,e,f$.
Then we have
$$S_N=\sum_{(n,k)\in\mathbb{N}^2, f(n,k)\leq N }1=\sum_{n\geq 0}\frac{1}{a6^n+c+(-1)^ne}N+o(N).$$
\end{lemma}
\begin{proof}
We denote $u_n=a6^n+c+(-1)^ne, v_n=b6^n+d+f(-1)^n$ and we remark that $f(n,k)=ku_n+v_n$.  Then we define the following number $m$: 
$$m=max\{n\in\mathbb{N}, v_n \leq N\}.$$
Now we remark that $f(n,k)\leq N$ implies $k\leq \frac{N-v_n}{u_n}$, then we have
$$S_N=\sum_{n\leq m}(\lfloor  \frac{N-v_n}{u_n}\rfloor+1).$$
$$S_N=\sum_{n\leq m}\lfloor  \frac{N-v_n}{u_n}\rfloor+m.$$
We can remark that the inequality $v_n\leq N$ gives 
$b6^n\leq N-d-f$, thus we deduce 
$m\leq \frac{\ln{N}}{b\ln{6}}$. Finally we can write

$$S_N=\sum_{n\leq m}\frac{N-v_n}{u_n}+O(\ln{N}).$$

Remark that
Now we have $f(n,k)=k(a6^n+c+(-1)^ne)+b6^n+d+f(-1)^n$. 
$$\frac{N-v_n}{u_n}=\frac{N-b6^n-d-f(-1)^n}{a6^n+c+(-1)^ne}=\frac{N}{a6^n+c+(-1)^ne}+0(1).$$
We deduce finally
$$S_N=\sum_{n\leq m}\frac{N}{a6^n+c+(-1)^ne}+O(\ln{N}).$$

Now we remark that $\displaystyle\lim_{N\rightarrow+\infty}m=+\infty$, moreover the numerical serie of general term $\frac{1}{a6^n+c+(-1)^ne}$ is convergent. Thus we deduce that $$\sum_{n\leq m}\frac{1}{a6^n+c+(-1)^ne}=\sum_{n}\frac{1}{a6^n+c+(-1)^ne}+o(1).$$
We deduce

$$S_N=\sum_{n\geq 0}\frac{1}{a6^n+c+(-1)^ne}N+o(N).$$
\end{proof}

\begin{lemma}\label{cal-comp-fin}
There exists $\beta>0$ such that 
$$\displaystyle\sum_{i=0}^Nb(i)\sim \beta N.$$
Moreover we can give a formula for $\beta$:
\begin{align*}
\beta=&\frac{14}{15}+\\
&\displaystyle\sum_{n\geq 0}(\frac{7}{48.6^n.2+14+2(-1)^n}+\frac{7}{18.6^n.2+14-(-1)^n})\\
&-\displaystyle\sum_{n\geq 0}(\frac{7}{78.6^n.2+14+5(-1)^n}+\frac{7}{48.6^n.2+14-5(-1)^n}).
\end{align*}

\end{lemma}
\begin{proof}
First we remark that $i(v)$ for a bispecial word $v$ can be equal to $1$ or $-1$. Thus we have
$$\displaystyle\sum_{i=0}^Nb(i)=\displaystyle\sum_{i\leq N}[\sum_{v\in\mathcal{BS}(i)}1-\displaystyle\sum_{v\in\mathcal{BW}(i)}1].$$
 We will first consider the subset of strong bispecial words. The other part is similar.
The preceding proposition shows that the length of the bispecial words is always of the form $f(n,k)$ with $a,b,c,d,e$ rational numbers. Then we will consider only one familly and denote by $S_N$ the associated sum. W need to count the number of bispecial words of length less than $N$. To do this we count the number of integers $(k,n)$ such that $(ak+b)6^n+c_k+d\leq N$. The preceding Lemma gives the result.

We have six sequences of values for $a,b$. We sum them, and we deduce that the number of strong bispecial words of length less than $N$ is:
$$\displaystyle\sum_{i\leq N}\sum_{v\in\mathcal{BS}(i)}1=\beta_1 N+o(N),$$
where $\beta_1$ is the sum of the different constants depending on the six values of $a,b,c,d,e$.

For the weak bispecial words, the lengths have the same form, thus the calculus is similar. We obtain another constant $\beta_2$. It suffices to remark that we have $\beta_1>\beta_2$, and we define $\beta=\beta_1-\beta_2$.
$$\displaystyle\sum_{i\leq N}\sum_{v\in\mathcal{BW}(i)}1=\beta_2 N+o(N).$$
We find

\begin{align*}
\beta_1=&\frac{1}{4}+\frac{1}{6}+\frac{1}{8}+\frac{1}{6}+\frac{1}{8}+\frac{1}{2}+\\
&2[\displaystyle\sum_{n\geq 0}\frac{35}{48.6^n.20+140+20(-1)^n}+\displaystyle\sum_{n\geq 0}\frac{35}{18.6^n.20+140-10(-1)^n}].
\end{align*}

$$\beta_1=\frac{4}{3}+\displaystyle\sum_{n\geq 0}\frac{7}{48.6^n.2+14+2(-1)^n}+\displaystyle\sum_{n\geq 0}\frac{7}{18.6^n.2+14-(-1)^n}.$$

$$\beta_2=\frac{4}{10}+2\displaystyle\sum_{n\geq 0}[\frac{35}{78.6^n.20+140+50(-1)^n}+\frac{35}{48.6^n.20+140-50(-1)^n}].$$
$$\beta_2=\frac{2}{5}+\displaystyle\sum_{n\geq 0}[\frac{7}{78.6^n.2+14+5(-1)^n}+\frac{7}{48.6^n.2+14-5(-1)^n}].$$
\end{proof}

\begin{corollary}
We deduce that $p(n)\sim \frac{\beta n^2}{2}.$
\end{corollary}
\begin{proof}
The proof is a direct consequence of Lemma \ref{julien} and Lemma \ref{cal-comp-fin}.
\end{proof}


\section{Regular decagon}\label{secdeca}
In this short section we explain how to deal with the case of the regular decagon. In fact this case can be deduced easily from the case of the regular pentagon. In Figure \ref{decaT} the lengths are not correct, but the angles have correct values. We have drawn the partition and its image by $\hat{T}$.

\subsection{Induction}
\begin{lemma}
The map $\hat{T}_{deca}$ is defined on five sets. The definitions of these sets are the following, see Figure  \ref{decaT} :
\begin{itemize}
\item The first set $V_1$ is a triangle, and $\hat{T}_{deca}$ is a rotation of angle $4\pi/5$ on this set.
\item The second set $V_2$ is a quadrilateral, and $\hat{T}_{deca}$ is a rotation of angle $3\pi/5$ on this set.
\item The third set is $V_3$ a quadrilateral, and $\hat{T}_{deca}$ is a rotation of angle $2\pi/5$ on this set.
\item The fourth set $V_4$ is an infinite polygon with four edges, and $\hat{T}_{deca}$ is a rotation of angle $\pi/5$ on this set.
\item The last one $V_5$ is a cone, and $\hat{T}_{deca}$ is a translation on this set.
\end{itemize}
\end{lemma}

\begin{figure}
\begin{center}  
\includegraphics[width= 6cm]{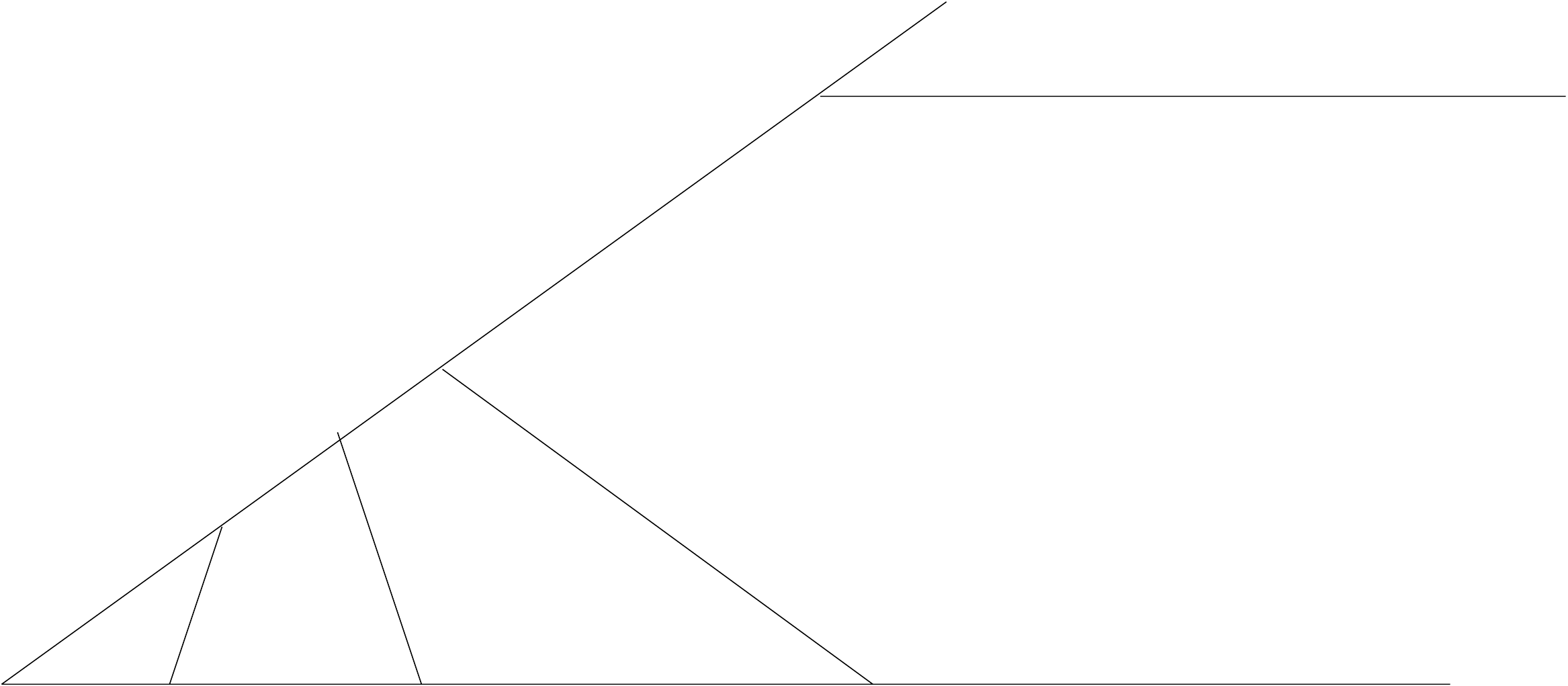}
\includegraphics[width= 6cm]{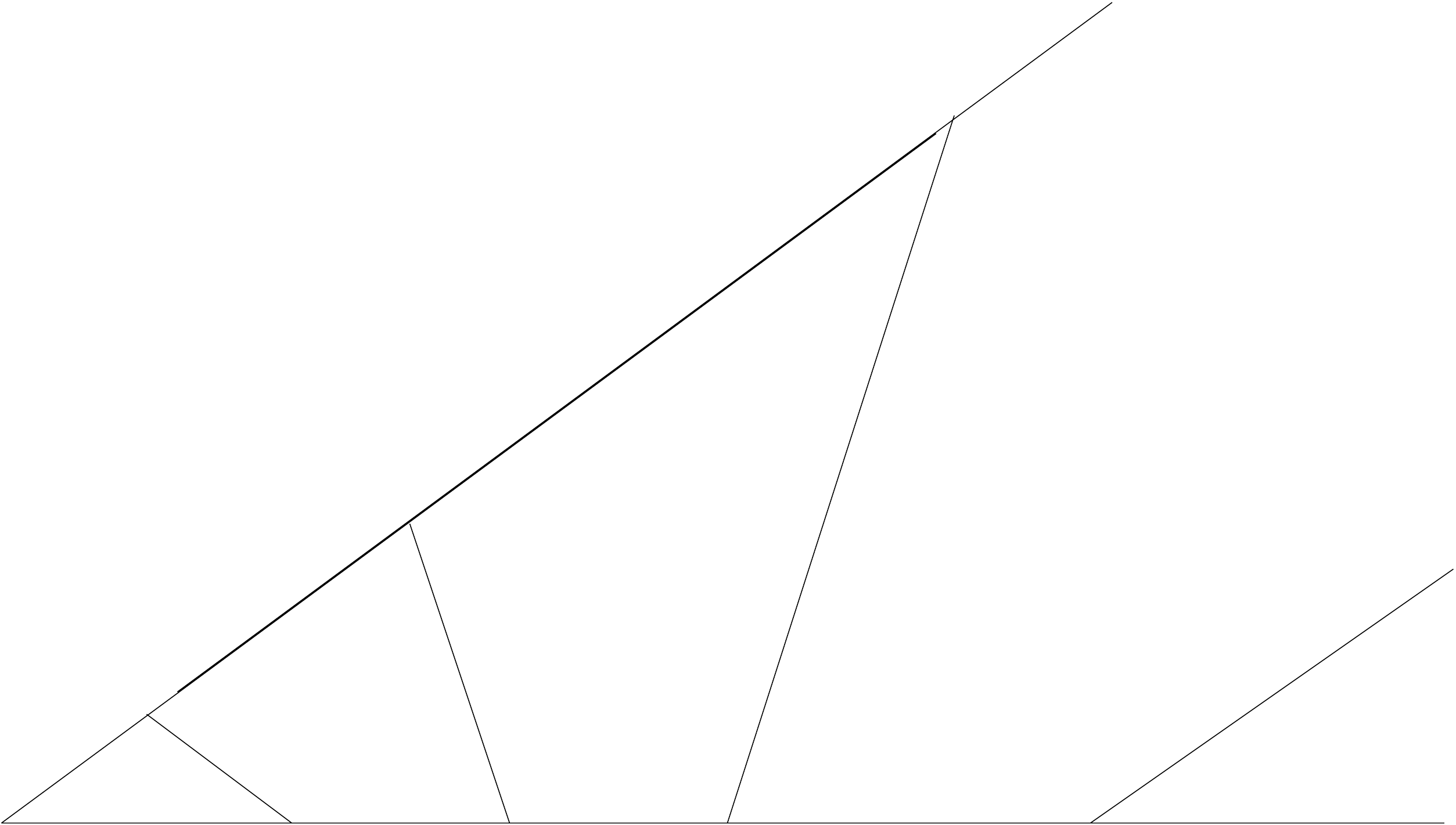}

\caption{Definition of $\hat{T}$ for the decagon} \label{decaT}
\end{center}
\end{figure}

\begin{lemma}\label{ret-deca}
Consider the maps $\hat{T}_{penta}, \hat{T}_{deca}$ related to the outer billiard map outside the regular pentagon respectively the regular decagon. Then consider the induced map on $U_3$, and denote it by $\hat{T}_{penta,3}$. Then there exists a translation $s$ such that
$$\hat{T}_{deca}=s^{-1}\circ \hat{T}_{penta,3}\circ s.$$  
\end{lemma}

The results follow from Lemma \ref{ret-deca}.
\begin{corollary}
There exists a bijective map $\theta$ between the coding of the decagon and the pentagon which is 
$$
\begin{array}{cccccc}
\theta&:&\mathcal{L}'_{deca}&\rightarrow&\mathcal{L}'_{penta}\cap \rho(U_3)\\
&&\begin{cases}1\\ 2\\ 3\\4\\5\end{cases}& \rightarrow& \begin{cases}322222\\32222\\3222\\ 322\\ 32\end{cases}
\end{array}
$$
\end{corollary}

\subsection{Results}
We use the same method as for the pentagon, and we deduce the language of the outer billiard map outside the regular decagon, and the complexity function.
\begin{theorem}
There exists a constant $C>0$ such that
$$p(n)\sim Cn^2.$$
\end{theorem}

\bibliographystyle{alpha}
\bibliography{biblidual}
\end{document}